\documentclass{amsart}
\usepackage{amsfonts,amssymb,amsmath,amsthm,mathrsfs}
\usepackage{url}
\usepackage{enumerate}

\newtheorem{theorem}{Theorem}[section]
\newtheorem{lemma}[theorem]{Lemma}

\newtheorem{corollary}[theorem]{Corollary}
\theoremstyle{definition}
\newtheorem{definition}[theorem]{Definition}
\newtheorem{remark}[theorem]{Remark}
\numberwithin{equation}{section}
\newtheorem{assumption}[theorem]{Assumption}

\begin{document}

\title[multilinear singular integral operators]{Weighted vector-valued bounds for a class of multilinear singular integral operators and applications}

\author[J. Chen]{Jiecheng Chen}
\address{
Jiecheng Chen, Department of  Mathematics,
Zhejiang Normal University,
Jinhua  321004,
P. R. China
}
\author[G. Hu]{Guoen Hu}
\address{Guoen Hu, Department of  Applied Mathematics, Zhengzhou Information Science and Technology Institute,
Zhengzhou 450001,
P. R. China}
\email{jcchen@zjnu.edu.cn, guoenxx@163.com}
\thanks{The research of the first author was supported by the NNSF of
China under grant $\#$11271330, and the research of the second (corresponding) author was supported by
the NNSF of
China under grant $\#$11371370.}


\begin{abstract}
{In this paper, we investigate the weighted vector-valued bounds for a class of multilinear singular integral operators, and its commutators, from
$L^{p_1}(l^{q_1};\,\mathbb{R}^n,w_1)\times\dots\times L^{p_m}(l^{q_m};\,\mathbb{R}^n,w_m)$ to $L^{p}(l^q;\,\mathbb{R}^n,\nu_{\vec{w}})$, with $p_1,\dots,p_m$, $q_1,\,\dots,\,q_m\in (1,\,\infty)$, $1/p=1/p_1+\dots+1/p_m$, $1/q=1/q_1+\dots+1/q_m$ and $\vec{w}=(w_1,\,\dots,\,w_m)$  a multiple $A_{\vec{P}}$ weights. Our argument also leads to the weighted weak type endpoint estimates for the commutators.
}\end{abstract}

\maketitle


\section{Introduction}
In his remarkable work \cite{mu}, Muckenhoupt characterized the class of weights $w$ such that $M$, the Hardy-Littlewood maximal operator, satisfies  the weighted $L^p$  $(p\in (1,\,\infty))$ estimate
\begin{eqnarray}\|Mf\|_{L^{p,\,\infty}(\mathbb{R}^n,\,w)}\lesssim \|f\|_{L^p(\mathbb{R}^n,\,w)}.\end{eqnarray}
The inequality (1.1) holds if and only if $w$ satisfies the $A_p(\mathbb{R}^n)$ condition, that is,
$$[w]_{A_p}:=\sup_{Q}\Big(\frac{1}{|Q|}\int_Qw(x){\rm d}x\Big)\Big(\frac{1}{|Q|}\int_{Q}w^{-\frac{1}{p-1}}(x){\rm d}x\Big)^{p-1}<\infty,$$
where the  supremum is taken over all cubes in $\mathbb{R}^n$, $[w]_{A_p}$ is called the $A_p$ constant of $w$. Also, Muckenhoupt proved that $M$ is bounded on $L^p(\mathbb{R}^n,\,w)$ if and only if $w$ satisfies the $A_p(\mathbb{R}^n)$ condition. Since then, considerable attention has been paid to the theory of $A_p(\mathbb{R}^n)$ and the weighted norm inequalities with $A_p(\mathbb{R}^n)$ weights for main operators in Harmonic Analysis, see \cite[Chapter 9]{gra} and related references therein.

However, the classical results on the weighted norm inequalities with $A_p(\mathbb{R}^n)$ weights did not reflect the quantitative dependence of the $L^p(\mathbb{R}^n,\,w)$ operator norm in terms of the relevant constant involving the weights. The question of the sharp dependence of the weighted estimates in terms of the $A_p(\mathbb{R}^n)$ constant specifically raised by Buckley \cite{bu}, who proved that if $p\in (1,\,\infty)$ and $w\in A_{p}(\mathbb{R}^n)$, then
\begin{eqnarray}\|Mf\|_{L^{p}(\mathbb{R}^n,\,w)}\lesssim_{n,\,p}[w]_{A_p}^{\frac{1}{p-1}}\|f\|_{L^{p}(\mathbb{R}^n,\,w)}.\end{eqnarray}
Moreover, the estimate (1.2) is sharp since the exponent $1/(p-1)$ can not be replaced by a smaller one. Hyt\"onen and P\'erez \cite{hp} improved the estimate (1.4), and showed that
\begin{eqnarray}\|Mf\|_{L^{p}(\mathbb{R}^n,\,w)}\lesssim_{n,\,p}\big([w]_{A_p}[w^{-\frac{1}{p-1}}]_{A_{\infty}}\big)^{\frac{1}{p}}\|f\|_{L^{p}(\mathbb{R}^n,\,w)}.\end{eqnarray}
where and in the following, for a weight $u$, $[u]_{A_{\infty}}$ is defined by
$$[u]_{A_{\infty}}=\sup_{Q\subset \mathbb{R}^n}\frac{1}{u(Q)}\int_{Q}M(u\chi_Q)(x){\rm d}x.$$
It is well known that for $w\in A_p(\mathbb{R}^n)$, $[w^{-\frac{1}{p-1}}]_{A_{\infty}}\lesssim [w]_{A_p}^{\frac{1}{p-1}}$. Thus, (1.3) is more subtle than (1.2).

The sharp dependence of the weighted estimates of  singular integral operators in terms of the $A_p(\mathbb{R}^n)$ constant  was much more complicated.  Petermichl \cite{pet1,pet2} solved this question for Hilbert transform and Riesz transform.   Hyt\"onen \cite{hyt}
proved that  for a  Calder\'on-Zygmund operator $T$ and $w\in A_2(\mathbb{R}^n)$,
\begin{eqnarray}\|Tf\|_{L^{2}(\mathbb{R}^n,\,w)}\lesssim_{n}[w]_{A_2}\|f\|_{L^{2}(\mathbb{R}^n,\,w)}.\end{eqnarray}
This solved  the so-called $A_2$ conjecture. Combining the estimate (1.4) and the extrapolation theorem in \cite{dra}, we know that
for a Calder\'on-Zygmund operator $T$, $p\in (1,\,\infty)$ and $w\in A_p(\mathbb{R}^n)$,
\begin{eqnarray}\|Tf\|_{L^{p}(\mathbb{R}^n,\,w)}\lesssim_{n,\,p}[w]_{A_p}^{\max\{1,\,\frac{1}{p-1}\}}\|f\|_{L^{p}(\mathbb{R}^n,\,w)}.\end{eqnarray}
In \cite{ler2}, Lerner  gave a much simplier proof of (1.5) by  controlling the Calder\'on-Zygmund operator using sparse operators.

Let $K(x;\,y_1,\,\dots,\,y_m)$ be a locally integrable function
defined away from the diagonal $x=y_1=\dots=y_m$ in $\mathbb{R}^{mn}$. An
operator $T$ defined on
$\mathcal{S}(\mathbb{R}^n)\times\dots\times\mathcal{S}(\mathbb{R}^n)$ (Schwartz space)
and taking values in $\mathcal{S}'(\mathbb{R}^n)$, is said to be an
$m$-multilinear singular integral operator with kernel $K$, if $T$ is
$m$-multilinear, and satisfies that
\begin{equation}T(f_1,\,\dots,f_m)(x)=\int_{\mathbb{R}^{mn}}K(x;\,y_1,\,\dots,y_m)\prod_{j=1}^mf_j(y_j){\rm
d}y_1\dots dy_{m},\end{equation} for bounded functions $f_1,\dots,f_m$
with compact supports, and $x\in\mathbb{R}^n\backslash \cap_{j=1}^m{\rm
supp}\,f_j$. Operators of this type were originated in the
remarkable works of Coifman and Meyer \cite{cm1}, \cite{cm2}, and
are useful in multilinear analysis. We say that $T$ is an $m$-linear Calder\'on-Zygmund operator, if $T$ is bounded from $L^{r_1}(\mathbb{R}^n)\times \dots\times L^{r_m}(\mathbb{R}^n)$ to $L^r(\mathbb{R}^n)$ for some $r_1,\,\dots,\,r_m\in (1,\,\infty)$ and $r\in (1/m,\,\infty)$ with $1/r=1/r_1+\dots+1/r_m$, and $K$ is
a multilinear Calder\'on-Zygmund kernel,   that is, $K$ satisfies the size
condition that for all $(x,\,y_1, \dots,y_m)\in \mathbb{R}^{(m+1)n}$ with $x\not
=y_j$ for some $1\leq j\leq m$,
\begin{equation}
|K(x;\,y_1, \dots,y_m)|\lesssim\frac{1}{\left(\sum_{j=1}^m|x-y_j|\right)^{mn}}
\end{equation}
and satisfies the  regularity condition that for some $\alpha\in (0,\,1]$
$$
|K(x;\,y_1, \dots,y_m)-K(x';\,y_1, \dots,y_m)|\lesssim
\frac{|x-x'|^{\alpha}}{\left(\sum_{j=1}^m|x-y_j|\right)^{mn+\alpha}}
$$
whenever $\max_{1\leq k\leq m}|x-y_k|\ge 2|x-x'|$, and for all $1\leq j\leq m$,
$$
\big|K(x;\,y_1, \dots,y_j\dots,y_m)
-K(x;\,y_1,\dots,y_j',\dots,y_m)\big| \lesssim\frac{|y_j-y'_j|^{\alpha}}{\left(\sum_{i=1}^m|x-y_i|\right)^{mn+\alpha}}
$$
whenever $\max_{1\leq k\leq m}|x-y_k|\ge 2|y_j-y_j'|$. Grafakos and Torres \cite{gt2}
considered the behavior of multilinear Calder\'on-Zygmund operators on
$L^{1}(\mathbb{R}^n)\times\dots\times L^{1}(\mathbb{R}^n)$, and established a $T1$ type theorem
for the operator $T$. To consider the weighted estimates for the multilinear Calder\'on-Zygmund operators, Lerner, Ombrossi, P\'erez, Torres and Trojillo-Gonzalez \cite{ler} introduced the following definition.
\begin{definition}\label{d1.1}
Let $m\in\mathbb{N}$, $w_1,\dots,w_m$ be
weights, $p_1,\dots,p_m\in[1,\,\infty)$, $p\in (0,\,\infty)$ with
$1/p=1/p_1+\dots+1/p_m$. Set $\vec{w}=(w_1,\,\dots,\,w_m)$,
$\vec{P}=(p_1,\,...,\,p_m)$ and
$\nu_{\vec{w}}=\prod_{k=1}^{m}w_k^{p/p_k}$. We say that $\vec{w}\in
A_{\vec{P}}(\mathbb{R}^{mn})$ if
$$[\vec{w}]_{A_{\vec{P}}}=\sup_{Q\subset
\mathbb{R}^n}\Big(\frac{1}{|Q|}\int_Q\nu_{\vec{w}}(x)\,{\rm
d}x\Big)
\prod_{k=1}^m\Big(\frac{1}{|Q|}\int_Qw_k^{-\frac{1}{{p_k}-1}}(x)\,{\rm
d}x\Big)^{p/p_k'}<\infty,$$ when
$p_k=1$,
$\Big(\frac{1}{|Q|}\int_Qw_k^{-\frac{1}{p_k-1}}(x)\,{\rm
d}x\Big)^{1-1/p_k}$ is understood as $(\inf_{Q}w_k\big)^{-1}$.\end{definition}

Lerner et al. \cite{ler} proved that if $p_1,\,\dots,\,p_m\in [1,\,\infty)$ and $p\in [1/m,\,\infty)$ with $1/p=1/p_1+\dots+1/p_m$,  and $\vec{w}=(w_1,\,\dots,\,w_m)\in A_{\vec{P}}(\mathbb{R}^{mn})$, then an $m$-linear Calder\'on-Zygmund operator $T$ is bounded
from $L^{p_1}(\mathbb{R}^n,\,w_1)\times \dots\times L^{p_m}(\mathbb{R}^n,\,w_m)$ to $L^{p,\,\infty}(\mathbb{R}^n, \nu_{\vec{w}})$, and when $\min_{1\leq j\le m}p_j>1$, $T$ is bounded from $L^{p_1}(\mathbb{R}^n, w_1)\times \dots\times L^{p_m}(\mathbb{R}^n,w_m)$ to $L^{p}(\mathbb{R}^n,\,\nu_{\vec{w}})$.
Li, Moen and Sun \cite{lms} considered the sharp dependence of the weighted estimates of  multilinear Calder\'on-Zygmund operators in terms of the $A_{\vec{P}}(\mathbb{R}^{mn})$ constant, and proved that
\begin{theorem}\label{t1.0}
Let $T$ be an $m$-linear Calder\'on-Zygmund operator, $p_1,\dots,p_m\in (1,\,\infty)$, $p\in [1,\,\infty)$ such that $1/p=1/p_1+\dots+1/p_m$, $\vec{w}=(w_1,\,\dots,\,w_m)\in A_{\vec{P}}(\mathbb{R}^{mn})$. Then \begin{eqnarray}\qquad\|T(f_1^k,\dots,f_m^k)\|_{L^p(\mathbb{R}^n,\nu_{\vec{w}})}\lesssim[\vec w]_{A_{\vec P}}^{\max \{1, \frac{p_1'}{p},\cdots, \frac{p_m'}{p}\}}\prod_{j=1}^m\|f_j\|_{L^{p_j}
(\mathbb{R}^n,w_j)}.\end{eqnarray}
Moreover, the exponent on $[\vec w]_{A_{\vec P}}$ is sharp.
\end{theorem}
Conde-Alongso and Rey \cite{car} proved that the conclusion in Theorem \ref{t1.0} is still true for the case $p\in (1/m,\,1)$.
For other works about the weighted estimates of multilinear Calder\'on-Zygmund operators, see  \cite{ls,bcd,dhli} and references therein.

To consider the mapping properties for the commutator of Calder\' on, Duong, Grafakos and Yan \cite{dgy} introduced a class
of multilinear singular integral operators via the following generalized approximation to the identity.
\begin{definition}
A family of operators $\{A_t\}_{t>0}$ is said to be an approximation to the identity, if for
every $t>0$, $A_t$ can be represented by the kernel at in the following sense: for every function $u\in L^p(\mathbb{R}^n)$
with $p\in [1,\,\infty]$ and a. e. $x\in\mathbb{R}^n$,
$$A_tu(x)=\int_{\mathbb{R}^n}a_t(x,\,y)u(y){\rm d}y,$$
and the kernel $a_t$ satisfies that for all $x,\,y\in\mathbb{R}^n$ and $t>0$,
\begin{eqnarray}|a_t(x,\,y)|\le h_t(x,\,y)=t^{-n/s}h\Big(\frac{|x-y|}{t^{1/s}}\Big),\end{eqnarray}
where $s>0$ is a constant and $h$ is a positive, bounded and decreasing function such that for some constant $\eta>0$,
\begin{eqnarray}\lim_{r\rightarrow\infty}r^{n+\eta}h(r)=0.\end{eqnarray}
\end{definition}

\begin{assumption}\label{a1.1}
For each fixed $j$ with $1\leq j\leq m$, there exists an approximation to the identity
$\{A_t^j\}_{t>0}$ with kernels $\{a_t^j(x,\,y)\}_{t>0}$, and there exist kernels $K_t^j(x;\,y_1,\dots,\,y_m)$, such that for bounded
functions $f_1,\,\dots,\,f_m$ with compact supports, and $x\in\mathbb{R}^n\backslash\cap_{k=1}^m{\rm supp}\,f_k$,
$$T(f_1,\dots,f_{j-1},A_t^jf_j,f_{j+1}\dots,f_m)(x)=
\int_{\mathbb{R}^{nm}}K_t^j(x;\,y_1,\dots,y_m)
\prod_{k=1}^mf_k(y_k){\rm d}\vec{y},$$
and there exists a function $\phi\in C(\mathbb{R})$ with ${\rm supp}\,\phi\subset[-1,\,1]$, and a constant $\varepsilon\in (0,\,1]$, such that for all $x,\,y_1,\,\dots,\,y_m\in\mathbb{R}^n$ and all $t>0$ with $2t^{1/s}\leq |x-y_j|$,
\begin{eqnarray*}&&|K(x;\,y_1,\,\dots,\,y_m)-K_t^j(x;\,y_1,\,\dots,\,y_m)|\\
&&\quad\lesssim\frac{t^{\varepsilon/s}}{(\sum_{k=1}^m|x-y_k|)^{mn+\varepsilon}}+\frac{1}{(\sum_{k=1}^m|x-y_k|)^{mn}}\sum_{1\leq i\leq m,\,i\not =j}\phi\Big(\frac{|y_i-y_j|}{t^{1/s}}\Big).
\end{eqnarray*}
\end{assumption}
As it was pointed out in \cite{dgy}, an operator with such  a kernel is called a multilinear singular integral operator with non-smooth kernel, since the kernel $K$  may enjoy  no smoothness in the variables $y_1\,\dots,y_m$. Also, it was pointed out in \cite{dgy} that if
$T$ is an $m$-linear Calder\'on-Zygmund operator, then $T$ also satisfies Assumption \ref{a1.1}.
Duong, Grafakos and Yan \cite{dgy} proved that if $T$ satisfies Assumption \ref{a1.1}, and is bounded from $L^{r_1}(\mathbb{R}^n)\times \dots\times L^{r_m}(\mathbb{R}^n)$ to $L^{r,\,\infty}(\mathbb{R}^n)$ for some $r_1,\,\dots,\,r_m\in (1,\,\infty)$ and $r\in (1/m,\,\infty)$ with $1/r=1/r_1+\dots+1/r_m$, then $T$ is also bounded from $L^{1}(\mathbb{R}^n)\times \dots\times L^{1}(\mathbb{R}^n)$ to $L^{1/m,\,\infty}(\mathbb{R}^n)$. Recently, Hu and Li \cite{huli} considered the mapping properties from  $L^{1}(l^{q_1};\,\mathbb{R}^n)\times \dots\times L^{1}(l^{q_m};\,\mathbb{R}^n)$ to $L^{1/m,\,\infty}(l^q;\,\mathbb{R}^n)$ for the multilinear operator which satisfies
Assumption \ref{a1.1}.

The first purpose of this paper is to give an extension of Theorem \ref{t1.0} to the operators satisfying Assumption \ref{a1.1}. We further assume the kernel $K$ satisfies the following regularity condition: for $x,\, x',\, y_1,\,\dots,\,y_m\in\mathbb{R}^n$ with $8|x-x'|<\min_{1\leq j\leq m}|x-y_j|$, and each number $D$ such that $2|x-x'|<D$ and $4D<\min_{1\leq j\leq m}|x-y_j|$,
\begin{eqnarray}\quad|K(x; y_1,\,\dots,\, y_m)-K(
x';\,y_1,\,\dots,\,y_m)| \lesssim\frac{D^{\gamma}}{
\big(\sum_{j=1}^m|x-y_j|)^{nm+\gamma}}.\end{eqnarray}
This condition was introduced  in \cite{huz}, in order to established the weighted estimates for multilinear singular integral operators with non-smooth kernels. As it was pointed out in \cite{huz}, the operators considered in \cite{dggy,gly} also satisfies Assumption \ref{a1.1} and (1.11). On the other hand, it is obvious that if $T$ is an $m$-linear Calder\'on-Zygmund operator, then $T$ also satisfies (1.11). Thus, the operators we consider here contain multilinear Calder\'on-Zygmund operators and multilinear
singular integral operators with non-smooth kernels.
To state our  results, we  first recall some notations.

Let $p,\,r\in(0,\,\infty]$ and $w$ be a weight. As usual, for a sequence of numbers $\{a_k\}_{k=1}^{\infty}$, we denote $\|\{a_k\}\|_{l^r}=\big(\sum_k|a_k|^r\big)^{1/r}$. The space $L^p(l^{r};\,\mathbb{R}^n,\,w)$ is defined as
$$L^p(l^{r};\,\mathbb{R}^n,\,w)=\big\{\{f_k\}_{k=1}^{\infty}:\, \|\{f_k\}\|_{L^p(l^r;\,\mathbb{R}^n,\,w)}<\infty\big\}$$
where
$$\|\{f_k\}\|_{L^p(l^r;\,\mathbb{R}^n,\,w)}=\Big(\int_{\mathbb{R}^n}\|\{f_k(x)\}\|_{l^r}^pw(x)\,{\rm d}x\Big)^{1/p}.$$
The space $L^{p,\,\infty}(l^{r};\,\mathbb{R}^n,\,w)$ is defined as
$$L^{p,\,\infty}(l^{r};\,\mathbb{R}^n,\,w)=\big\{\{f_k\}_{k=1}^{\infty}:\, \|\{f_k\}\|_{L^{p,\,\infty}(l^r;\,\mathbb{R}^n,\,w)}<\infty\big\}$$
with
$$\|\{f_k\}\|_{L^{p,\,\infty}(l^r;\,\mathbb{R}^n,\,w)}^p=\sup_{\lambda>0}\lambda^pw\Big(\Big\{x\in\mathbb{R}^n:\,
\|\{f_k(x)\}\|_{l^r}>\lambda\Big\}\Big).$$ When $w\equiv 1$, we denote  $\|\{f_k\}\|_{L^p(l^r;\,\mathbb{R}^n,\,w)}$ ($\|\{f_k\}\|_{L^{p,\infty}(l^r;\,\mathbb{R}^n,\,w)}$)
by $\|\{f_k\}\|_{L^p(l^r;\,\mathbb{R}^n)}$ ($\|\{f_k\}\|_{L^{p,\infty}(l^r;\,\mathbb{R}^n)}$) for simplicity.

Our first result can be stated as follows.
\begin{theorem}\label{t1.1}
Let $m\geq 2$, $T$ be an $m$-linear operator with kernel $K$  in the sense of (1.6), $r_1,\,\dots\,r_m\in (1,\,\infty)$, $r\in (0,\,\infty)$ such that $1/r=1/r_1+\dots+1/r_m$. Suppose that
\begin{itemize}
\item[\rm (i)]  $T$ is bounded from $L^{r_1}(\mathbb{R}^n)\times\dots\times L^{r_m}(\mathbb{R}^n)$ to $L^r(\mathbb{R}^n)$;
\item[\rm (ii)] The kernel $K$ satisfies size conditon (1.7) and regular condition (1.11);
\item[\rm (iii)] $T$ satisfies the Assumption \ref{a1.1}.
\end{itemize}
Let $p_1,\dots,p_m,\,q_1,\,\dots,q_m\in (1,\,\infty)$, $p,\,q\in (\frac{1}{m},\infty)$ such that $1/p=1/p_1+\dots+1/p_m$, $1/q=1/q_1+\dots+1/q_m$, $\vec{w}=(w_1,\,\dots,\,w_m)\in A_{\vec{P}}(\mathbb{R}^{mn})$. Then \begin{eqnarray}&&\|\{T(f_1^k,\dots,f_m^k)\}\|_{L^p(l^q;\mathbb{R}^n,\nu_{\vec{w}})}\lesssim[\vec w]_{A_{\vec P}}^{\max \{1, \frac{p_1'}{p},\cdots, \frac{p_m'}{p}\}}\prod_{j=1}^m\|\{f_j^k\}\|_{L^{p_j}
(l^{q_j};\mathbb{R}^n,w_j)}.\end{eqnarray}
\end{theorem}
\begin{remark}
As we pointed out,  operators in Theorem \ref{t1.1} contain multilinear Calder\'on-Zygmund operators as examples. This, together with the examples in \cite{lms}, shows that the estimate (1.12) is sharp.
\end{remark}
Now let $b$ be a locally integrable function. For $1\leq j\leq m$, define the commutator $[b,\,T]_j$ by
$$[b,\,T]_j(\vec{f})(x)=b(x)T(f_1,\,\dots,\,f_m)(x)-T(f_1,\,\dots,\,f_{j-1},bf_j,\,f_{j+1},\,\dots,\,f_m)(x).$$
Let $b_1,\,\dots,\,b_m$ be locally integrable functions and   $\vec{b}=(b_1,\,\dots,\,b_m)$. The multilinear commutator of $T$ and $\vec{b}$ is defined by
\begin{eqnarray}T_{\vec{b}}(f_1,\dots,f_m)(x)=\sum_{j=1}^m[b_j,\,T]_j(f_1,\dots,f_m)(x).\end{eqnarray}
As it was showed in \cite{chpp,hp,dhli}, by the conclusion (1.12), we can prove that, under the hypothesis of Theorem \ref{t1.1},  for $p_1,\,\dots,\,p_m,\,p\in (1,\,\infty)$ and $\vec{w}\in A_{\vec{P}}(\mathbb{R}^{mn})$,
\begin{eqnarray}\qquad\|T_{\vec{b}}(f_1,\dots,f_m)\|_{L^p(\mathbb{R}^n,\,\nu_{\vec{w}})}&\lesssim &\|\vec{b}\|_{{\rm BMO}(\mathbb{R}^n)}[\vec{w}]_{A_p}^{\max\{1,\frac{p_1'}{p},\dots,\frac{p_m'}{p}\}}\\
&&\times\big([\nu_{\vec{w}}]_{A_{\infty}}+\sum_{j=1}^{m}[\sigma_j]_{A_{\infty}}\big)\prod_{j=1}^m\|f_j\|_{L^{p_j}(\mathbb{R}^n,\,w_j)}.\nonumber
\end{eqnarray}
However, for the case of $p\in (0,\,1)$, we do not know if we can deduce the weighted estimate for $T_{\vec{b}}$ like (1.14) from (1.12), the argument used in \cite{chpp,hp,dhli} does not apply.

\begin{definition} Let $s\in[1,\,\infty)$. A measurable function $b$ on $\mathbb{R}^n$ is said to belong to the space $Osc_{{\rm exp}\,L^s}(\mathbb{R}^n)$, if $\|b\|_{Osc_{{\rm exp}\,L^s}(\mathbb{R}^n)}<\infty$, with$$\|b\|_{Osc_{{\rm exp}\,L^s}(\mathbb{R}^n)}=\inf\Big\{C>0:\,\sup_{Q\subset \mathbb{R}^n}\frac{1}{|Q|}\int_{Q}
{\rm exp}\Big(\frac{|b(x)-\langle b\rangle_Q|}{C}\Big)^s{\rm d}x\leq 2\Big\},$$where and in the following, $\langle b\rangle_Q=\frac{1}{|Q|}\int_{Q}b(y){\rm d}y.$
\end{definition}

For details of this space, see \cite{perez2}. We remark that $Osc_{{\rm exp}\,L^1}(\mathbb{R}^n)={\rm BMO}(\mathbb{R}^n).$

Our  result concerning the weighted bound of $T_{\vec{b}}$  can be stated as follows.

\begin{theorem}\label{t1.2} Let $T$ be an $m$-linear operator as in Theorem \ref{t1.1} and $T_{\vec{b}}$  the commutator defined by (1.13).  Then for $p_1,\,\dots,\,p_m;\,q_1,\,\dots,\,q_m\in (1,\,\infty)$, $p,\,q\in (1/m,\,\infty)$ with $1/p=1/p_1+\dots+1/p_m$ and $1/q=1/q_1+\dots+1/q_m$, and $\vec{w}\in A_{\vec{P}}(\mathbb{R}^{mn})$,
\begin{eqnarray}&&\|\{T_{\vec{b}}(f_1^k,\dots,f_m^k)\}\|_{L^p(l^q;\,\mathbb{R}^n,\nu_{\vec{w}})}\lesssim\big(\sum_{j=1}^m\|b_j\|_{Osc_{{\rm exp}L^{s_j}}(\mathbb{R}^n)}\big)[\vec w]_{A_{\vec P}}^{\max \{1, \frac{p_1'}{p},\cdots, \frac{p_m'}{p}\}}\\
&&\quad\qquad\times\Big([\nu_{\vec{w}}]_{A_{\infty}}^{\frac{1}{s_*}}+\sum_{i=1}^m[\sigma_i]_{A_{\infty}}^{\frac{1}{s_i}}\Big)
\prod_{j=1}^m\|f_j^k\|_{L^{p_j}
(l^{q_j};\,\mathbb{R}^n,w_j)},\nonumber\end{eqnarray}
where and in the following, $\sigma_j(x)=w_j^{-\frac{1}{p_j-1}}(x), $ $s_*=\min_{1\leq i\leq m}s_i$.
\end{theorem}
Our argument in the proof of Theorems \ref{t1.1} and \ref{t1.2} also leads to the following weighted weak type endpoint estimate of $T_{\vec{b}}$.
\begin{theorem}\label{t1.3}Let $T$ be an $m$-linear operator  in Theorem \ref{t1.1}, $b_j\in Osc_{{\rm exp}L^{s_j}}(\mathbb{R}^n)$ $(j=1,\,\dots,\,m)$ and $T_{\vec{b}}$ be the commutator defined by (1.13).  Then for $q_1,\,\dots,\,q_m\in (1,\,\infty)$, $q\in (1/m,\,\infty)$ with $1/q=1/q_1+\dots+1/q_m$, $\vec{w}\in A_{{1,\,\dots,\,1}}(\mathbb{R}^{mn})$ and $\lambda>0$,
\begin{eqnarray}&&\nu_{\vec{w}}(\{x\in\mathbb{R}^n:\,\|\{T_{\vec{b}}(f_1^k,\dots,f_m^k)(x)\}\|_{l^q}>\lambda\})\\
&&\quad\lesssim\prod_{j=1}^m\Big(\int_{\mathbb{R}^{n}}
\frac{\|\{f_j^k(y_j)\}\|_{l^{q_j}}}{\lambda^{\frac{1}{m}}}\log^{\frac{1}{s_*}}\Big(1+\frac{\|\{f_j^k(y_j)\}\|_{l^{q_j}}}
{\lambda^{\frac{1}{m}}}\Big)w_j(y_j){\rm d}y_j\Big)^{\frac{1}{m}}.\nonumber\end{eqnarray}
\end{theorem}
\begin{remark} For the case that $T$ is multilinear Calder\'on-Zygmund operator and $b_1,\,\dots,\,b_m\in {\rm BMO}(\mathbb{R}^n)$,  (1.16) (the case $\{f_j^k\}=\{f_j\}$) was proved in \cite{ler}. Although Bui and Duong \cite{buiduong} considered the weighted estimate for $T_{\vec{b}}$ under the hypothesis of Theorem \ref{t1.1}, the argument in \cite{ler} does not leads to the conclusion in Theorem \ref{t1.3}.
\end{remark}

In what follows, $C$ always denotes a
positive constant that is independent of the main parameters
involved but whose value may differ from line to line. We use the
symbol $A\lesssim B$ to denote that there exists a positive constant
$C$ such that $A\le CB$.  Constant with subscript such as $C_1$,
does not change in different occurrences. For any set $E\subset\mathbb{R}^n$,
$\chi_E$ denotes its characteristic function.  For a cube
$Q\subset\mathbb{R}^n$ and $\lambda\in(0,\,\infty)$, we use $\ell(Q)$ (${\rm diam}Q$) to denote the side length (diamter) of $Q$, and
$\lambda Q$ to denote the cube with the same center as $Q$ and whose
side length is $\lambda$ times that of $Q$. For $x\in\mathbb{R}^n$ and $r>0$, $B(x,\,r)$ denotes the ball centered at $x$ and having radius $r$.

\section{Preliminaries}
Recall that  the standard dyadic grid in $\mathbb{R}^n$ consists of all cubes of the form $$2^{-k}([0,\,1)^n+j),\,k\in \mathbb{Z},\,\,j\in\mathbb{Z}^n.$$
Denote the standard dyadic grid by $\mathcal{D}$. For a fixed cube $Q$, denote by $\mathcal{D}(Q)$ the set of dyadic cubes with respect to $Q$, that is, the cubes from $\mathcal{D}(Q)$ are formed by repeating subdivision of $Q$ and each of descendants into $2^n$ congruent subcubes.

As usual, by a general dyadic grid $\mathscr{D}$,  we mean a collection of cube with the following properties: (i) for any cube $Q\in \mathscr{D}$, it side length $\ell(Q)$ is of the form $2^k$ for some $k\in \mathbb{Z}$; (ii) for any cubes $Q_1,\,Q_2\in \mathscr{D}$, $Q_1\cap Q_2\in\{Q_1,\,Q_2,\,\emptyset\}$; (iii) for each $k\in \mathbb{Z}$, the cubes of side length $2^k$ form a partition of $\mathbb{R}^n$.

Let $\mathcal{S}$ be a family of cubes and $\eta\in (0,\,1)$. We say that $\mathcal{S}$ is $\eta$-sparse,  if, for each fixed $Q\in \mathcal{S}$, there exists a measurable subset $E_Q\subset Q$, such that $|E_Q|\geq \eta|Q|$ and $\{E_{Q}\}$ are pairwise disjoint. A family is called simply  sparse if $\eta=1/2$.

For constants $\beta_1,\,\dots,\,\beta_m\in [0,\,\infty)$, let $\vec{\beta}=(\beta_1,\,\dots,\,\beta_m)$. Associated with  the sparse family $\mathcal{S}$ and  $\vec{\beta}$, we define the sparse operator $\mathcal{A}_{m;\,\mathcal{S},L(\log L)^{\vec{\beta}}}$  by
\begin{eqnarray}\mathcal{A}_{m;\,\mathcal{S},\,L(\log L)^{\vec{\beta}}}(f_1,\dots,f_m)(x)=\sum_{Q\in\mathcal{S}}\prod_{j=1}^m\|f_j\|_{L(\log L)^{\beta_j},\,Q}\chi_{Q}(x),\end{eqnarray}
with
$$\|f_j\|_{L(\log L)^{\beta_j},\,Q}=\Big\{\lambda>0:\,\frac{1}{|Q|}\int_{Q}\frac{|f(y)|}{\lambda}\log^{\beta_j}\Big(1+\frac{|f(y)|}{\lambda}\Big){\rm d}y\leq 1\Big\}.$$
For the case of $\vec{\beta}=(0,\dots,0)$, we denote $\mathcal{A}_{m;\,\mathcal{S},\,L(\log L)^{\vec{\beta}}}$ by $\mathcal{A}_{m;\,\mathcal{S}}$ for simplicity. Also, we denote $\mathcal{A}_{1;\,\mathcal{S},\,L(\log L)^{\beta}}$ ($\mathcal{A}_{1;\,\mathcal{S}}$) by $\mathcal{A}_{\mathcal{S},\,L(\log L)^{\beta}}$ ($\mathcal{A}_{\mathcal{S}}$). For a weight $u$, let $$\langle h\rangle_{Q}^{u}=\frac{1}{u(Q)}\int_Qh(y)u(y){\rm d}y,$$ and
\begin{eqnarray}&&\widetilde{\mathcal{A}}_{m;\mathcal{S}}(f_1,\dots,f_m)(x)=\mathcal{A}_{m;\mathcal{S}}(f_1\sigma_1,\dots,f_m\sigma_m)(x)=\sum_{Q\in\mathcal{S}}\prod_{j=1}^m\langle f_j\rangle_{Q}^{\sigma_j}\langle\sigma_j\rangle_Q\chi_Q(x).
\end{eqnarray}
For a dyadic grid $\mathscr{D}$, and sparse family $\mathcal{S}\subset \mathscr{D}$, it was proved in \cite{lms} that for $p_1,\,\dots,\,p_m\in (1,\,\infty)$, $p\in (0,\,\infty)$ such that $1/p=1/p_1+\dots+1/p_m$, $\vec{w}=(w_1,\,\dots,\,w_m)\in A_{\vec{P}}(\mathbb{R}^{mn})$, and $\sigma_j=w_j^{-\frac{1}{p_j-1}}$ with $j=1,\,\dots,\,m$,
\begin{eqnarray}\quad\|\widetilde{\mathcal{A}}_{m;\,\mathcal{S}}(f_1,\dots,f_m)\|_{L^p(\mathbb{R}^n,\,\nu_{\vec{w}})}\lesssim [\vec{w}]_{A_p}^{\max\{1,\,\frac{p_1'}{p},\,\dots,\,\frac{p_m'}{p}\}}
\prod_{j=1}^m\|f_j\|_{L^{p_j}(\mathbb{R}^n,\,\sigma_j)},\end{eqnarray}
and so
\begin{eqnarray}\quad\|\mathcal{A}_{m;\,\mathcal{S}}(f_1,\dots,f_m)\|_{L^p(\mathbb{R}^n,\,\nu_{\vec{w}})}\lesssim [\vec{w}]_{A_p}^{\max\{1,\,\frac{p_1'}{p},\,\dots,\,\frac{p_m'}{p}\}}
\prod_{j=1}^m\|f_j\|_{L^{p_j}(\mathbb{R}^n,\,w_j)}.\end{eqnarray}

\begin{theorem}\label{t2.1}Let $p_1,\,\dots,\,p_m\in (1,\,\infty)$, $p\in (0,\infty)$ such that $1/p=1/p_1+\dots+1/p_m$, and $\vec{w}=(w_1,\,\dots,\,w_m)\in A_{\vec{P}}(\mathbb{R}^{mn})$. Set $\sigma_i=w^{-1/(p_i-1)}$. Let $\mathscr{D}$ be a dyadic grid and $\mathcal{S}\subset \mathscr{D}$ be a sparse family. Then
for $\beta_1,\,\dots,\,\beta_m\in [0,\,\infty)$,
\begin{eqnarray}&&\|\mathcal{A}_{m;\mathcal{S},L(\log L)^{\vec{\beta}}}(f_1,\dots,f_m)\|_{L^p(\mathbb{R}^n,\nu_{\vec{w}})}\\
&&\quad\lesssim [\vec{w}]_{A_p}^{\max\{1,\frac{p_1'}{p},\dots,\frac{p_m'}{p}\}}\prod_{j=1}^m[\sigma_j]_{A_{\infty}}^{\beta_j}
\|f_j\|_{L^{p_j}(\mathbb{R}^n,w_j)}.\nonumber\end{eqnarray}
\end{theorem}

\begin{proof}We employ the ideas used in the proof of Theorem 3.2 in \cite{lms}, in which Theorem \ref{t2.1} was proved for the case of $\beta_1=\beta_2=0$, see also the proof of Theorem B in \cite{bcd}.
As it is well known, $\vec{w}\in A_{\vec{P}}(\mathbb{R}^{mn})$ implies $\sigma_j=w_j^{-\frac{1}{p_j-1}}\in A_{mp_j'}(\mathbb{R}^n)$ (see \cite{ler}). Also, it was pointed out in \cite{hp} that for the constant $\theta_{\sigma}=1+\frac{1}{\tau_n[\sigma]_{A_{\infty}}}$ with $\tau_n=2^{11+n}$,
\begin{eqnarray}\Big(\frac{1}{|Q|}\int_Q\sigma_j^{r_{\sigma_j}}(x){\rm d}x\Big)^{\frac{1}{r_{\sigma_j}}}\le 2\frac{1}{|Q|}\int_Q\sigma_j(x){\rm d}x.\end{eqnarray}
Let $\varrho_j=(1+p_j)/2$. We can verify that
$$\big\|\sigma_j^{\frac{1}{\varrho_j'}}\big\|_{L^{\varrho_j'}(\log L)^{\varrho_j'\beta_j},\,Q}\lesssim\|\sigma_j\|_{L(\log L)^{\varrho_j' \beta_j},\,Q}^{\frac{1}{\varrho_j'}}.
$$
Recall that \begin{eqnarray}\|h\|_{L(\log L)^{\varrho},\,Q}\lesssim\max\big\{1,\,\frac{1}{(\delta-1)^{\varrho}}\big\}\Big(\frac{1}{|Q|}\int_{Q}|h(y)|^{\delta}{\rm d}y\Big)^{\frac{1}{\delta}}.\end{eqnarray}
It then follows  that
\begin{eqnarray*}
\big\|\sigma_j^{\frac{1}{\varrho_j'}}\big\|_{L^{\varrho_j'}(\log L)^{\varrho_j'\beta_j},\,Q}&\lesssim&\frac{1}{(r_{\sigma_j}-1)^{\beta_j}}\Big(\frac{1}{|Q|}\int_{Q}\sigma_j^{r_{\sigma_j}}(y){\rm d}y\Big)^{\frac{1}{\varrho_j' r_{\sigma_j}}}\\
&\lesssim&[\sigma_j]_{A_{\infty}}^{\beta_j}\Big(\frac{1}{|Q|}\int_{Q}\sigma_j(y){\rm d}y\Big)^{\frac{1}{\varrho_j'}}.\end{eqnarray*}
Applying the generalization of H\"older's inequality (see \cite{rr}), we deduce that
\begin{eqnarray}
\|f_j\sigma_j\|_{L(\log L)^{\beta_j},\,Q}&\lesssim &
\Big(\frac{1}{|Q|}\int_{Q}|f_j|^{\varrho_j}\sigma_j\Big)^{\frac{1}{\varrho_j}}
\|\sigma_j^{\frac{1}{\varrho_j'}}\|_{L^{\varrho_j'}(\log L)^{\varrho_j'\beta_j},\,Q}\\
&\lesssim&[\sigma_j]_{A_{\infty}}^{\beta_j}\Big(\frac{1}{|Q|}\int_{Q}|f_j|^{\varrho_j}\sigma_j
\Big)^{\frac{1}{\varrho_j}}\Big(\frac{1}{|Q|}\int_{Q}\sigma_j\Big)^{\frac{1}{\varrho_j'}}\nonumber\\
&=&[\sigma_j]_{A_{\infty}}^{\beta_j}\Big(\frac{1}{\sigma_j(Q)}\int_{Q}|f_j|^{\varrho_j}\sigma_j
\Big)^{\frac{1}{\varrho_j}}\frac{\sigma_j(Q)}{|Q|}\nonumber\\
&\lesssim&[\sigma_j]_{A_{\infty}}^{\beta_j}\langle M_{\sigma_j,\,\varrho_j}^{\mathscr{D}}f_j\rangle_{Q}^{\sigma_j}\langle \sigma_j\rangle_Q,\nonumber
\end{eqnarray}
here and in the following, $M_{\sigma_j,\,\varrho_j}^{\mathscr{D}}$ is the maximal operator defined by
$$M_{\sigma_j,\,\varrho_j}^{\mathscr{D}}f_j(x)=\sup_{I\ni x,\,I\in \mathscr{D}}\Big(\frac{1}{\sigma_j(I)}\int_{I}|f_j(y)|^{\varrho_j}\sigma_j(y){\rm d}y
\Big)^{\frac{1}{\varrho_j}}.$$
We then deduce that
$$\prod_{j=1}^m\|f_j\sigma_j\|_{L(\log L)^{\beta_j},\,Q}\lesssim\prod_{i=1}^m[\sigma_i]_{A_{\infty}}^{\beta_i}\sum_{Q\in\mathcal{S}}\prod_{j=1}^m\langle M_{\sigma_j,\,\varrho_j}^{\mathscr{D}}f_j\rangle_{Q}^{\sigma_j}\langle\sigma_j\rangle_Q.
$$
This, via the estimate (2.3) and the fact that $M_{\sigma_j,\,\varrho_j}^{\mathscr{D}}$ is bounded on $L^{p_j}(\mathbb{R}^n,\,\sigma_j)$ with bounds independent of $\sigma_j$, yields
\begin{eqnarray*}
&&\Big\|\sum_{Q\in\mathcal{S}}\prod_{j=1}^m\|f_j\sigma_j\|_{L(\log L)^{\beta_j},\,Q}\chi_{Q}\Big\|_{L^p(\mathbb{R}^n,\,\nu_{\vec{w}})}\\
&&\quad\lesssim \prod_{i=1}^m[\sigma_i]_{A_{\infty}}^{\beta_i}\Big\|\sum_{Q\in\mathcal{S}}\prod_{j=1}^m\langle M_{\sigma_j,\,\varrho_j}f_j\rangle_{Q}^{\sigma_j}\langle\sigma_j\rangle_Q\chi_Q\Big\|_{L^p(\mathbb{R}^n,\,\nu_{\vec{w}})}\\
&&\quad\lesssim[\vec{w}]_{A_p}^{\max\{1,\,\frac{p_1'}{p},\,\dots,\,\frac{p_m'}{p}\}}\prod_{i=1}^m[\sigma_i]_{A_{\infty}}^{\beta_i}
\prod_{j=1}^m\|M_{\sigma_j,\,\varrho_j}^{\mathscr{D}}f_j\|_{L^{p_j}(\mathbb{R}^n,\,\sigma_j)}
\end{eqnarray*}
and then completes the proof of Theorem \ref{t2.1}.
\end{proof}
For locally integrable functions $b_1,\,\dots,\,b_m$ and a sparse family $\mathcal{S}$, let
\begin{eqnarray}\mathcal{A}_{m;\,\mathcal{S},\,\vec{b}}(f_1,\,\dots,\,f_m)(x)=\sum_{Q\in \mathcal{S}}\Big(\sum_{i=1}^m|b_i(x)-\langle b_i\rangle_Q|\Big)\prod_{j=1}^m\langle f_j\rangle_Q\chi_{Q}(x).\end{eqnarray}

\begin{theorem}\label{t2.2}Let $p_1,\,\dots,\,p_m\in (1,\,\infty)$, $p\in (0,\,\infty)$ such that $1/p=1/p_1+\dots+1/p_m$, and $\vec{w}=(w_1,\,\dots,\,w_m)\in A_{\vec{P}}(\mathbb{R}^{mn})$.  Let $\mathscr{D}$ be a dyadic grid and $\mathcal{S}\subset \mathscr{D}$ be a sparse family, $b_i\in Osc_{{\rm exp}L^{s_i}}(\mathbb{R}^n)$ ($s_i\in [1,\,\infty)$) with $\sum_{i=1}^m\|b_i\|_{{Osc_{{\rm exp}L^{s_i}}(\mathbb{R}^n)}}=1$. Then
\begin{eqnarray}&&\|\mathcal{A}_{m;\mathcal{S},\vec{b}}(f_1,\dots,f_m)\|_{L^p(\mathbb{R}^n,\nu_{\vec{w}})}\lesssim [\vec{w}]_{A_p}^{\max\{1,\frac{p_1'}{p},\dots,\frac{p_m'}{p}\}}
[\nu_{\vec{w}}]_{A_{\infty}}^{\frac{1}{s_*}}\prod_{j=1}^m\|f_j\|_{L^{p_j}(\mathbb{R}^n,w_j)}.\end{eqnarray}

\end{theorem}
\begin{proof}We first consider the case of $p\in (0,\,1]$. Write
\begin{eqnarray*}
\int_{\mathbb{R}^n}\big(\mathcal{A}_{m,\,\mathcal{S},\,\vec{b}}\vec{f}(x)\big)^p\nu_{\vec{w}}(x){\rm d}x
&\leq &\sum_{Q\in \mathcal{S}}\prod_{j=1}^m\langle |f_j|\rangle_Q^p\sum_{i=1}^m\int_{\mathbb{R}^n}|b_i(x)-\langle b_i\rangle_Q|^p\nu_{\vec{w}}(x){\rm d}x\\
&\leq&\sum_{Q\in \mathcal{S}}\prod_{j=1}^m\langle |f_j|\rangle_Q^p\sum_{i=1}^m|Q|\|\nu_{\vec{w}}\|_{L(\log L)^{\frac{p}{s_i}},\,Q}\\
&\leq &[\nu_{\vec{w}}]_{A_{\infty}}^{\frac{p}{s_*}}\sum_{Q\in \mathcal{S}}\prod_{j=1}^m\langle |f_j|\rangle_Q^p\nu_{\vec{w}}(Q),
\end{eqnarray*}
where in the last inequality, we have invoked the estimates (2.7) and (2.6) for $\nu_{\vec{w}}$. It was proved in \cite[pp. 757-758]{lms} that
$$\sum_{Q\in \mathcal{S}}\prod_{j=1}^m\langle |f_j|\rangle_Q^p\nu_{\vec{w}}(Q)\lesssim [\vec{w}]_{A_p}^{\max\{p_1',\,\dots,\,p_m'\}}\prod_{j=1}^m\|f_j\|_{L^{p_j}(\mathbb{R}^n,\,w_j)}^p.
$$
The inequality (2.10) then follows in this case.

To consider the case of  $p\in (1,\,\infty)$, let $\varrho=\frac{1+p'}{2}$ with $p'=\frac{p}{p-1}$. Observe that by (2.7),
\begin{eqnarray*}
\|g\nu_{\vec{w}}\|_{L(\log L)^{\frac{1}{s_*}},\,Q}&\lesssim&\Big(\frac{1}{|Q|}\int_{Q}|g(x)|^{\varrho}\nu_{\vec{w}}(x){\rm d}x\Big)^{\frac{1}{\varrho}}\|\nu_{\vec{w}}^{\frac{1}{\varrho'}}\|_{L^{\varrho'}(\log L)^{\frac{\varrho'}{s_*}},\,Q}\\
&\lesssim&[w]_{A_{\infty}}^{\frac{1}{s_*}}\Big(\frac{1}{\nu_{\vec{w}}(Q)}\int_{Q}|g(x)|^{\varrho}\nu_{\vec{w}}(x){\rm d}x\Big)^{\frac{1}{\varrho}}\frac{\nu_{\vec{w}}(Q)}{|Q|}.
\end{eqnarray*}
Therefore, by the generalization of H\"older's inequality (see \cite{rr}),
\begin{eqnarray*}
&&\int_{\mathbb{R}^n}\mathcal{A}_{m,\,\mathcal{S},\,\vec{b}}(f_1,\dots,f_m)(x)g(x)\nu_{\vec{w}}(x){\rm d}x\\
&&\quad= \sum_{Q\in \mathcal{S}}\prod_{j=1}^m\langle |f_j|\rangle_Q\sum_{i=1}^m\int_{\mathbb{R}^n}|b_i(x)-\langle b_i\rangle_Q|g(x)\nu_{\vec{w}}(x){\rm d}x\\
&&\quad\leq\sum_{Q\in \mathcal{S}}\prod_{j=1}^m\langle |f_j|\rangle_Q|Q|\|g\nu_{\vec{w}}\|_{L(\log L)^{\frac{1}{s_*}},\,Q}\\
&&\quad\leq [\nu_{\vec{w}}]_{A_{\infty}}^{\frac{1}{s_*}}\sum_{Q\in \mathcal{S}}\prod_{j=1}^m\langle |f_j|\rangle_Q\inf_{x\in Q}M_{\nu_{\vec{w}},\,\varrho}^{\mathscr{D}}g(x)\nu_{\vec{w}}(Q)\\
&&\quad\leq[\nu_{\vec{w}}]_{A_{\infty}}^{\frac{1}{s_*}}\|\mathcal{A}_{\mathcal{S}}(f_1,\,\dots,\,f_m)\|_{L^p(\mathbb{R}^n,\,\nu_{\vec{w}})}
\|M_{\nu_{\vec{w}},\,\varrho}^{\mathscr{D}}g\|_{L^{p'}(\mathbb{R}^n,\,\nu_{\vec{w}})}.
\end{eqnarray*}
Our desired conclusion then follows from (2.4) and the $L^{p'}(\mathbb{R}^n,\,\nu_{\vec{w}})$ boundedness of $M_{\nu_{\vec{w}},\,\varrho}^{\mathscr{D}}$.
\end{proof}

\section{Proof of Theorems \ref{t1.1} and \ref{t1.2}}
Let $T$ be an $m$-sublinear operator. Associated with $T$, let
$$\mathcal{M}_{T}(f_1,\dots,f_m)(x)=\sup_{Q\ni x}\big\|T(f_1,\dots,f_m)(\xi)-T(f_1\chi_{3Q},\dots,f_m\chi_{3Q})(\xi)\big\|_{L^{\infty}(Q)}.
$$
Following the argument in \cite{ler2}, we have
\begin{lemma}\label{l3.1}
Let $q_1,\,\dots,\,q_m\in (1,\,\infty)$, $q\in (1/m,\,\infty)$ such that $1/q=1/q_1+\dots+1/q_m$, $T$ be an $m$-sublinear operator which is bounded from $L^1(l^{q_1};\,\mathbb{R}^n)\times\dots\times L^1(l^{q_m};\,\mathbb{R}^n)$ to $L^{\frac{1}{m},\,\infty}(l^{q};\,\mathbb{R}^n)$. Then for any cube $Q_0$ and a. e. $x\in Q_0$, we have that
\begin{eqnarray*}\|\{T(f_1^k\chi_{3Q_0},\,\dots,\,f_m^k\chi_{3Q_0})(x)\}\|_{l^q}&\leq & C_1\prod_{j=1}^m\|\{f_j^k(x)\}\|_{l^{q_j}}\\
&&+\big\|\{\mathcal{M}_{T}(f_1^k\chi_{3Q_0},\,\dots,\,f_m^k\chi_{3Q_0})(x)\}\big\|_{l^q},\end{eqnarray*}
provided that $\|\{f_1^k\}\|_{l^{q_1}},\,\dots,\,\|\{f_m^k\}\|_{l^{q_m}}\in L_{loc}^1(\mathbb{R}^n)$.
\end{lemma}
\begin{proof} We follow the line  in \cite{ler3}. Let $x\in {\rm int}Q_0$  be a point of approximation continuity of $\|\{T(f_1\chi_{3Q_0},\,\dots,\,f_m\chi_{3Q_0})\}\|_{l^q}$.
For  $r,\,\epsilon>0$, the set
\begin{eqnarray*}E_r(x)&=&\{y\in B(x,\,r):\, \Big|\|\{T(f_1^k\chi_{3Q_0},\,\dots,\,f_m^k\chi_{3Q_0})(x)\}\|_{l^q}\\
&&\qquad-\|\{T(f_1^k\chi_{3Q_0},\dots,\,f_m^k\chi_{3Q_0})(y)\}\|_{l^q}\Big|<\epsilon\}\end{eqnarray*}
satisfies that
$\lim_{r\rightarrow 0}\frac{|E_r(x)|}{|B(x,\,r)|}=1.$ Denote by $Q(x,\,r)$  the smallest cube centered at $x$ and containing $B(x,\,r)$.
Let $r>0$  small enough such that $Q(x,\,r)\subset Q_0$. Then for $y\in E_r(x)$,
\begin{eqnarray*}\|\{T(f_1^k\chi_{3Q_0},\dots,f_m^k\chi_{3Q_0})(x)\}\|_{l^q}&<&\|\{T(f_1^k\chi_{3Q_0},\,\dots,\,f_m^k\chi_{3Q_0})(y)\}\|_{l^q}+\epsilon\\
&\leq& \|\{T(f_1^k\chi_{3Q(x,\,r)},\,\dots,\,f_m^k\chi_{3Q(x,\,r)})(y)\}\|_{l^q}\\
&&+\big\|\{\mathcal{M}_{T}(f_1^k\chi_{3Q_0},\,\dots,\,f_m^k\chi_{3Q_0})(x)\}\big\|_{l^q}+\epsilon.\end{eqnarray*}
Thus, for  $\varsigma\in (0,\,1/m)$,
\begin{eqnarray*}&&\big\|\big\{T(f_1^k\chi_{3Q_0},\,\dots,\,f_m^k\chi_{3Q_0})(x)\big\}\big\|_{l^q}\\
&&\quad\leq \Big(\frac{1}{|E_{s}(x)|}\int_{E_s(x)}\|\{T(f_1^k\chi_{3Q(x,r)},\,\dots,\,f_m^k\chi_{3Q(x,r)})(y)\}\|_{l^q}^{\varsigma}{\rm d}y\Big)^{\frac{1}{\varsigma}}\\
&&\qquad+\big\|\{\mathcal{M}_{T}(f_1^k\chi_{3Q_0},\,\dots,f_m^k\chi_{3Q_0})(x)\}\big\|_{l^q}+\epsilon\\
&&\quad\leq C\prod_{j=1}^m\langle\|\{f_j^k\}\|_{l^{q_j}}\rangle_{Q(x,\,r)}+\big\|\{\mathcal{M}_{T}(f_1^k\chi_{3Q_0},\,\dots,\,f_m^k\chi_{3Q_0})(x)\}
\big\|_{l^q}+\epsilon,\end{eqnarray*}
since $T$ is bounded from $L^1(l^{q_1};\,\mathbb{R}^n)\times\dots\times L^1(l^{q_m};\,\mathbb{R}^n)$ to $L^{\frac{1}{m},\,\infty}(l^{q};\,\mathbb{R}^n)$. Taking $r\rightarrow 0+$ then leads to the conclusion (i).
\end{proof}
\begin{lemma}\label{l3.0}
Let $\tau\in (0,\,1)$ and $M_{\tau}$ be the maximal operator defined by$$M_{\tau}f(x)=\big(M(|f|^{\tau})(x)\big)^{\frac{1}{\tau}}.$$
Then for any $p\in (\tau,\,\infty)$ and $u\in A_{p/\tau}(\mathbb{R}^n)$
$$u(\{x\in\mathbb{R}^n:\|\{M_{\tau}f_k(x)\}\|_{l^q}>\lambda)\lesssim_{u,p}\lambda^{-p}\sup_{t\geq C\lambda}t^pu(\{x\in\mathbb{R}^n:\|\{f_k(x)\}\|_{l^q}>t\}).
$$
\end{lemma}
\begin{proof} For each fixed $\lambda>0$, decompose $f_k$ as
$$f_k(y)=f_k(y)\chi_{\{\|\{f_k(y)\}\|_{l^q}\leq \lambda\}}(y)+f_k(y)\chi_{\{\|\{f_k(y)\}\|_{l^q}>\lambda\}}(y):=f_k^1(y)+f_k^2(y).$$It then follows that
$$
u(\{x\in\mathbb{R}^n:\,\|\{M_{\tau}f_k(x)\}\|_{l^q}>2^{\frac{1}{\tau}}\lambda\})
\leq u(\{x\in\mathbb{R}^n:\,\|\{M(|f_k^2|^{\tau})(x)\}
\|_{l^{\frac{q}{\tau}}}>\lambda^{\tau}\}).
$$
Recall that $u\in A_{p/\tau}$ implies that $u\in A_{\frac{p-\epsilon}{\tau}}(\mathbb{R}^n)$ for some $\epsilon\in (0,\,p-\tau)$, and that $M$ is
bounded on $L^{\frac{p-\epsilon}{\tau}}(l^q;\,\mathbb{R}^n,\,u)$ (see \cite{fes}). Therefore,
\begin{eqnarray*}
&&u(\{x\in\mathbb{R}^n:\|\{M(|f_k^2|^{\tau})(x)\}
\|_{l^{\frac{q}{\tau}}}>\lambda^{\tau}\})\\
&&\quad\lesssim\lambda^{-p+\epsilon}\int_{\mathbb{R}^n}\|\{f_k^2(x)\}\|_{l^q}^{p-\epsilon}u(x){\rm d}x\\
&&\quad\lesssim u(\{x\in\mathbb{R}^n:\,\|\{f_k\}\|_{l^q}>\lambda\})\\
&&\qquad+\lambda^{-p+\epsilon}\int_{\lambda}^{\infty}u(\{x\in\mathbb{R}^n:\,\|\{f_k^2(x)\}\|_{l^q}>t\})t^{p-\epsilon-1}{\rm d}t\\
&&\quad\lesssim\lambda^{-p}\sup_{t\geq \lambda}t^pu(\{x\in\mathbb{R}^n:\,\|\{f_k(x)\}\|_{l^q}>t\}).
\end{eqnarray*}
This yields our desired conclusion.
\end{proof}

\begin{lemma}\label{l3.2}
Let $q_1,\,\dots, q_m\in (1,\,\infty)$, $q\in (1/m,\,\infty)$ such that $1/q=1/q_1+\dots+1/q_m$. Under the hypothesis of Theorem \ref{t1.1}, the operator $\mathcal{M}_{T}$ is bounded from $L^1(l^{q_1};\,\mathbb{R}^n)\times\dots\times L^1(l^{q_m};\,\mathbb{R}^n)$ to $L^{\frac{1}{m},\,\infty}(l^{q};\,\mathbb{R}^n)$.
\end{lemma}
\begin{proof}  For simplicity, we only consider the bilinear case, namely, $m=2$. For  $\epsilon>0$, let
$${T}^{\epsilon}(f_1,\,f_2)(x)=\int_{\max_{j}|x-y_j|>\epsilon}K(x;\,y_1,\,y_2)f_1(y_1)f_2(y_2){\rm d}y_1{\rm d}y_2.$$
We claim that for each $\tau\in (0,\,1/2)$,
\begin{eqnarray}
\sup_{\epsilon>0}|T^{\epsilon}(f_1,\,f_2)(x)|\lesssim
M_{\tau}(T(f_1,\,f_2))(x)+Mf_1(x)Mf_2(x).
\end{eqnarray}
To prove this, we will employ the ideas used in \cite{dgy,gly}. let
$$G^{\epsilon}(f_1,\,f_2)(x,\,z)=\int_{\min_{j}|x-y_j|>\epsilon}K(z;\,y_1,\,y_2)f_1(y_1)f_2(y_2){\rm d}y_1{\rm d}y_2.$$
For  functions $f_1,\,\dots,\,f_m$, set
$$f_j^{(0)}(y)=f_j(y)\chi_{B(x,\,\epsilon)}(y).$$ Let
$$A_{\epsilon}(f_1,\,f_2)(x)=\int_{\max_{j}|x-y_j|>\epsilon,\atop{\min_{j}|x-y_j|\le\epsilon}}\big|K(x;\,y_1,\,y_2)\big|
\big|f_1(y_1)f_2(y_2)\big|{\rm d}y_1{\rm d}y_2,$$ and
$$E_{\epsilon}(f_1,\,f_2)(x,\,z)=\int_{\max_{j}|x-y_j|>\epsilon,\atop{\min_{j}|x-y_j|\le\epsilon}}\big|K(z;\,y_1,\,y_2)\big|
\big|f_1(y_1)f_2(y_2)\big|{\rm d}y_1{\rm d}y_2.$$
By the size condition, it is easy to verify that
\begin{eqnarray*}
A_{\epsilon}(f_1,\,f_2)(x)\lesssim Mf_1(x)Mf_2(x).
\end{eqnarray*}
Also, for $z\in B(x,\,\epsilon/8)$, we have
\begin{eqnarray*}
E_{\epsilon}(f_1,\,f_2)(x,\,z)\lesssim  Mf_1(x)Mf_2(x).
\end{eqnarray*}
It then follows from (1.11) that for  $z\in B(x,\,\epsilon/8)$,
\begin{eqnarray*}
&&\big|{T}^{\epsilon}(f_1,\,f_2)(x)-G^{\epsilon}(f_1,\,f_2)(x,\,z)\big|\\
&&\quad\lesssim A_{\epsilon}(f_1,\,f_2)(x)+E_{\epsilon}(f_1,\,f_2)(x,\,z)\\
&&\qquad+ \int_{\min_{i}|x-y_i|>\epsilon}\big|K(x;y_1,\,y_2)-K(z;y_1,\,y_2)\big|f_1(y_1)f_2(y_2){\rm d}y_1{\rm d}y_2\\
&&\quad\lesssim Mf_1(x)Mf_2(x).
\end{eqnarray*}
Observe that for
$z\in B(x,\,\epsilon/8)$,
\begin{eqnarray*}
|G^{\epsilon}(f_1,\,f_2)(x,z)|&\leq &\Big|\int_{\max_{j}|z-y_j|>\epsilon}K(z;\,y_1,\,y_2)f_1(y_1)f_2(y_2){\rm d}y_1{\rm d}y_2\Big|\\
&&+\int_{\frac{\epsilon}{2}\leq \max_{j}|x-y_j|\le 2\epsilon}\big|K(z;y_1,\,y_2)\big|\big|f_1(y_1)f_2(y_2)\big|{\rm d}y_1{\rm d}y_2\\
&\le &\big|T(f_1,\,f_2)(z)|+|T(f_1^{(0)},\,f_2^{(0)})(z)\big|+Mf_1(x)Mf_2(x).
\end{eqnarray*}
Therefore, for any $z\in B(x,\,\epsilon/8)$,
\begin{eqnarray*}
|{T}^{\epsilon}(f_1,\,f_2)(x)|\leq
|T(f_1,\,f_2)(z)|+|T(f_1^{(0)},\,f_2^{(0)})(z)|+\prod_{i=1}^2Mf_i(x).
\end{eqnarray*}This, together with the fact that $T$ is bounded from $L^{1}(\mathbb{R}^n)\times\dots\times
L^{1}(\mathbb{R}^n)$ to $L^{1/m,\,\infty}(\mathbb{R}^n)$, leads to (3.1).

Now let $${T}_{\epsilon}(f_1,\,f_2)(x)=\int_{\min_{j}|x-y_j|>\epsilon}K(x;\,y_1,\,y_2)f_1(y_1)f_2(y_2){\rm d}y_1{\rm d}y_2.$$
By the size condition (1.7), we see that
$$\big|{T}^{\epsilon}(f_1,\,f_2)(x)-T_{\epsilon}(f_1,\,f_2)(x)|\lesssim Mf_1(x)Mf_2(x).$$
and so$$\sup_{\epsilon>0}|T_{\epsilon}(f_1,\,f_2)(x)|\lesssim
M_{\tau}(T(f_1,\,f_2))(x)+Mf_1(x)Mf_2(x).$$

Let $Q\subset \mathbb{R}^n$ be a cube and $x,\,\xi\in Q$. Denote by $B_x$ the ball centered at $x$ and having diameter $10n{\rm diam}\,Q$. Then $3Q\subset B_x$. As in \cite{ler3}, we write
\begin{eqnarray*}&&\big|T(f_1\chi_{\mathbb{R}^n\backslash 3Q},f_2\chi_{\mathbb{R}^n\backslash 3Q})(\xi)|\\
&&\leq |T(f_1\chi_{\mathbb{R}^n\backslash B_x},f_2\chi_{\mathbb{R}^n\backslash B_x})(\xi)-T(f_1\chi_{\mathbb{R}^n\backslash B_x},f_2\chi_{\mathbb{R}^n\backslash B_x})(x)\big|+\sup_{\epsilon>0}|T_{\epsilon}(f_1,f_2)(x)|\\
&&\quad+|T(f_1\chi_{\mathbb{R}^n\backslash B_x},f_2\chi_{B_x\backslash 3Q})(\xi)\big|+|T(f_1\chi_{B_x\backslash 3Q},f_2\chi_{\mathbb{R}^n\backslash 3Q})(\xi)\big|
\end{eqnarray*}
It follows from the regularity condition (1.11) that
$$\big|T(f_1\chi_{\mathbb{R}^n\backslash B_x},\,f_2\chi_{\mathbb{R}^n\backslash B_x})(\xi)-T(f_1\chi_{\mathbb{R}^n\backslash B_x},\,f_2\chi_{\mathbb{R}^n\backslash B_x})(x)\big|\lesssim \prod_{i=1}^2Mf_i(x).$$
On the other hand, by the size condition (1.7), we have
\begin{eqnarray*}
\big|T(f_1\chi_{B_x\backslash 3Q},f_2\chi_{\mathbb{R}^n\backslash 3Q})(\xi)\big|&\lesssim&\int_{B_x}|f_1(y_1)|{\rm d}y_1
\int_{\mathbb{R}^n\backslash 3Q}\frac{|f_2(y_2)|}{(|x-y_2|+{\rm diam }Q)^{2n}}{\rm d}y_2\\
&\lesssim&Mf_1(x)Mf_2(x).
\end{eqnarray*}
Similarly,
\begin{eqnarray*}
\big|T(f_1\chi_{\mathbb{R}^n\backslash B_x},f_2\chi_{B_x\backslash 3Q})(\xi)\big|\lesssim Mf_1(x)Mf_2(x),
\end{eqnarray*}
and
$$\big|T(f_1\chi_{\mathbb{R}^n\backslash 3Q},\,f_2\chi_{3Q})(\xi)
+T(f_1\chi_{3Q},\,f_2\chi_{\mathbb{R}^n\backslash 3Q})(\xi)\big|\lesssim Mf_1(x)Mf_2(x).
$$
Combining the estimates above leads to that
\begin{eqnarray}\mathcal{M}_{T}(f_1,\,f_2)(x)\lesssim M_{\tau}(T(f_1,\,f_2))(x)+\prod_{i=1}^2Mf_i(x).
\end{eqnarray}

Recall that $T$ is bounded from $L^1(l^{q_1};\,\mathbb{R}^n)\times L^1(l^{q_2};\,\mathbb{R}^n)$ to $L^{\frac{1}{2},\,\infty}(l^{q};\,\mathbb{R}^n)$ (see \cite{huli}), and $M$ is bounded from $L^1(l^{q_j};\,\mathbb{R}^n)$ to $L^{1,\,\infty}(l^{q_j};\,\mathbb{R}^n)$ . Now we choose $\tau\in (0,\,1/2)$ in (3.2), our desired conclusion now follows from (3.2) and Lemma \ref{l3.0} immediately.
\end{proof}
\begin{theorem}\label{t3.1}

Let $q_1,\,\dots,q_m\in (1,\,\infty)$ and $q\in (1/m,\,\infty)$ with $1/q=1/q_1+\dots+1/q_m$. Suppose that both the operators $T$ and $\mathcal{M}_{T}$ are  bounded from $L^{1}(l^{q_1};\,\mathbb{R}^n)\times \dots\times L^{1}(l^{q_m};\,\mathbb{R}^n)$ to $L^{1/m,\,\infty}(l^q;\,\mathbb{R}^n)$. Then for $N\in\mathbb{N}$ and bounded functions $\{f_1^k\}_{1\leq k\leq N},\,\dots,\,\{f_m^k\}_{1\leq k\leq N}$ with compact supports, there exists a $\frac{1}{2}\frac{1}{3^n}$-sparse of family $\mathcal{S}$ such that for a. e. $x\in\mathbb{R}^n$,
\begin{eqnarray}\|\{T(f_1^k,\,\dots,\,f_m^k)(x)\}\|_{l^q}\lesssim \mathcal{A}_{m;\,\mathcal{S}}(\|\{f_1^k\}\|_{l^{q_1}},\,\dots,\|\{f_m^k\}\|_{l^{q_m}})(x).\end{eqnarray}
\end{theorem}

\begin{proof} Again, we only consider the case $m=2$. We follow the argument used  in \cite{ler3}.  At first, we claim that for each cube $Q_0\subset \mathbb{R}^n$,  there exist pairwise disjoint cubes $\{P_j\}\subset \mathcal{D}(Q_0)$, such that $\sum_{j}|P_j|\leq \frac{1}{2}|Q_0|$ and a. e. $x\in Q_0$,
\begin{eqnarray}
&&\big\|\big\{T(f_1^k\chi_{3Q_0},f_2^k\chi_{3Q_0})(x)\big\}\big\|_{l^q}\chi_{Q_0}(x)\\
&&\quad\le C\prod_{i=1}^2\langle\|\{f_i^k\}\|_{l^{q_i}}\rangle_{3Q_0}+\sum_{j}\|\{T(f_1^k\chi_{3P_j},f_2^k\chi_{3P_j})(x)\}\|_{l^{q}}\chi_{P_j}(x).\nonumber
\end{eqnarray}To prove this, let
$C_2>0$ which will be chosen later and \begin{eqnarray*}E&=&\big\{x\in Q_0:\,\|\{f_1^k(x)\}\|_{l^{q_1}}\|\{f_2^k(x)\}\|_{l^{q_2}}> \prod_{i=1}^2\langle\|\{f_i^k\}\|_{l^{q_i}}\rangle_{3Q_0}\big\}\\
&&\quad
\cup\big\{x\in Q_0:\,\|\{\mathcal{M}_{T}(f_1^k\chi_{3Q_0},\,f_2^k\chi_{3Q_0})(x)\}\|_{l^q}>C_2\langle\prod_{i=1}^2\langle \|\{f_i^k\}\|_{l^{q_i}}\rangle_{3Q_0}\big\}.
\end{eqnarray*}
If we choose $C_2$ large enough, we then know from Lemma \ref{l3.2} that $|E|\leq \frac{1}{2^{n+2}}|Q_0|.$ Now applying the Calder\'on-Zygmund decomposition to $\chi_E$ on $Q_0$ at level $\frac{1}{2^{n+1}}$, we then obtain a family of pairwise disjoint cubes $\{P_j\}$ such that
$$\frac{1}{2^{n+1}}|P_j|\leq |P_j\cap E|\leq \frac{1}{2}|P_j|,$$ and $|E\backslash \cup_jP_j|=0$. It then follows that $\sum_j|P_j|\leq \frac{1}{2}|E|$,  and $P_j\cap E^c\not =\emptyset.$ Therefore,
\begin{eqnarray}
&&\Big\|\big\|\{T(f_1^k\chi_{3Q_0\backslash 3P_j},\,f_2^k\chi_{3Q_0\backslash 3P_j})(\xi)\}\big\|_{l^q}
+\big\|\{T(f_1^k\chi_{3Q_0\backslash 3P_j},\,f_2^k\chi_{3P_j})(\xi)\}\big\|_{l^q}\\
&&\quad+\big\|\{T(f_1^k\chi_{3P_j},\,f_2^k\chi_{3Q_0\backslash 3P_j})(\xi)\}\big\|_{l^q}\Big\|_{L^{\infty}(P_j)}\leq C_2\prod_{i=1}^2\langle \|\{f_i^k\}\|_{l^{q_i}}\rangle_{3Q_0}.\nonumber
\end{eqnarray}
Note that
\begin{eqnarray}
&&\|\{T(f_1^k\chi_{3Q_0},f_2^k\chi_{3Q_0})(x)\}\|_{l^q}\chi_{Q_0}(x)\\
&&\quad\leq \|\{T(f_1^k\chi_{3Q_0},f_2^k\chi_{3Q_0})(x)\}\|_{l^q}\chi_{Q_0\backslash \cup_jP_j}(x)\nonumber\\
&&\qquad+\sum_j\|\{T(f_1^k\chi_{3P_j},\,f_2^k\chi_{3P_j})(x)\}\|_{l^q}\chi_{P_j}(x)+\sum_j{\rm D}_j(x)\chi_{P_j}(x),\nonumber\end{eqnarray}
with
\begin{eqnarray*}{\rm D}_j(x)&=&\|\{T(f_1^k\chi_{3Q_0\backslash 3P_j},f_2^k\chi_{3Q_0\backslash 3P_j})(x)\}\|_{l^q}+\|\{T(f_1^k\chi_{3Q_0\backslash 3P_j},f_2^k\chi_{3P_j})(x)\}\|_{l^q}\\
&&+\|\{T(f_1^k\chi_{3P_j},\,f_2^k\chi_{3Q_0\backslash 3P_j})(x)\}\|_{l^q}.
\end{eqnarray*}
(3.4) now follows from (3.5), (3.6) and Lemma \ref{l3.1}.

We can now conclude the proof of Theorem \ref{t3.1}. As it was proved in \cite{ler2}, the last estimate shows that there exists a $\frac{1}{2}$-sparse family $\mathcal{F}\subset \mathcal{D}(Q_0)$, such that for a. e. $x\in Q_0$,
\begin{eqnarray*}
\big\|\big\{T(f_1^k\chi_{3Q_0},\,f_2^k\chi_{3Q_0})(x)\big\}\big\|_{l^q}\chi_{Q_0}(x)\lesssim \sum_{Q\in \mathcal{F}}\prod_{i=1}^2\langle\|\{f_i^k\}\|_{l^{q_i}}\rangle_{3Q}\chi_{Q}(x).
\end{eqnarray*}
Recalling that $\{f_1^k\}_{1\leq k\leq N},\,\{f_2^k\}_{1\leq k\leq N}$ are functions in $L^1(\mathbb{R}^n)$ with compact supports, we can take now a partition of $\mathbb{R}^n$ by cubes $Q_j$ such that $\cup_{k=1}^N\cup_{i=1}^2{\rm supp}\,f_i^k\subset 3Q_j$
for each $j$ and obtain a $\frac{1}{2}$-
sparse family $\mathcal{F}_j\subset\mathcal{D}(Q_j)$ such that for a. e. $x\in Q_j$,
$$
\big\|\big\{T(f_1^k\chi_{3Q_j},\,f_2^k\chi_{3Q_j})(x)\big\}\big\|_{l^q}\chi_{Q_j}(x)\lesssim \sum_{Q\in \mathcal\mathcal{F}_j}\prod_{i=1}^2\langle\|\{f_i^k\}\|_{l^{q_i}}\rangle_{3Q}\chi_{Q}(x).
$$
 Setting $\mathcal{S}=\{3Q:\, Q\in\cup_j\mathcal{F}_j\}$, we see that (3.3) holds true for $\mathcal{S}$ and a. e. $x\in\mathbb{R}^n$.
\end{proof}
Similar to the proof of Theorem \ref{t3.1}, by mimicking the proof of Theorem 1.1 in \cite{ler3}, we can prove
\begin{theorem}\label{t3.2}
Let $q_1,\,\dots,q_m\in (1,\,\infty)$ and $q\in (1/m,\,\infty)$ with $1/q=1/q_1+\dots+1/q_m$, $b\in L_{\rm loc}^1(\mathbb{R}^n)$. Suppose that both the operators $T$ and $\mathcal{M}_{T}$ are  bounded from $L^{1}(l^{q_1};\,\mathbb{R}^n)\times \dots\times L^{1}(l^{q_m};\,\mathbb{R}^n)$ to $L^{1/m,\,\infty}(l^q;\,\mathbb{R}^n)$. Then for $N\in\mathbb{N}$ and bounded functions $\{f_1^k\}_{1\leq k\leq N}$,\,$\dots$, $\{f_m^k\}_{1\leq k\leq N}$ with compact supports, there exists a $\frac{1}{2}\frac{1}{3^n}$-sparse of family $\mathcal{S}$ such that for a. e. $x\in\mathbb{R}^n$,
\begin{eqnarray*}\|\{[b,T]_i(f_1^k,\dots,f_m^k)(x)\}\|_{l^q}&\lesssim&\sum_{Q\in\mathcal{S}}\langle|b-\langle b\rangle_Q|\|\{f_i^k\}\|_{l^{q_i}}\rangle_Q\prod_{j\not=i}\langle\|\{f_j^k\}\|_{l^{q_j}}\rangle_Q\chi_{Q}(x)\\
&&+\sum_{Q\in\mathcal{S}}|b(x)-\langle b\rangle_{Q}|\prod_{j=1}^m\langle \|\{f_j^k\}\|_{l^{q_j}}\rangle_{Q}\chi_{Q}(x).\nonumber\end{eqnarray*}
\end{theorem}

We are now ready to prove Theorem \ref{t1.1} and Theorem \ref{t1.2}.

\medskip

{\it Proof of Theorem \ref{t1.1}}. Obviously, it suffices to consider the case that $\{f_1^k\}$, $\dots$,\,$\{f_m^k\}$ are finite sequences.
By the  well known one-third trick (see \cite[Lemma 2.5]{hlp}), we know that if $\mathcal{S}$ is a sparse family, then there exist  general dyadic grids $\mathscr{D}_1,\,\dots,\mathscr{D}_{3^n}$, and sparse families $\mathcal{S}_i\subset \mathscr{D}_i$, with $i=1,\,\dots,\, 3^n$, such that
$$
\mathcal{A}_{m;\,\mathcal{S},\,L(\log L)^{\vec{\beta}}}(f_1,\,\dots,\,f_m)(x)\lesssim_{n}\sum_{i=1}^{3^n}\mathcal{A}_{m;\,\mathcal{S}_i,\,L(\log L)^{\vec{\beta}}}(f_1,\,\dots,\,f_m)(x).$$
Thus, Theorem \ref{t1.1} follows from Theorem \ref{t3.1}, Lemma \ref{l3.2} and the estimate (2.4) directly.\qed

\medskip

{\it Proof of Theorem \ref{t1.2}}. By the generalization of H\"older's inequality (see \cite{rr}), we know that
$$\big\langle |b_i(x)-\langle b_i\rangle_{Q}|\|\{f_i^k\}\|_{l^{q_i}}\big\rangle_{Q}\lesssim \big\|\|\{f_i^k\}\|_{l^{q_i}}\big\|_{L(\log L)^{\frac{1}{s_i}},\,Q}.
$$For $N\in\mathbb{N}$ and bounded functions $\{f_1^k\}_{1\leq k\leq N}$,\,$\dots$, $\{f_m^k\}_{1\leq k\leq N}$ with compact supports, we know from  Theorem \ref{t3.2} that there exists a $\frac{1}{2}\frac{1}{3^n}$-sparse of family $\mathcal{S}$ such that for a. e. $x\in\mathbb{R}^n$,
\begin{eqnarray}&&\|\{T_{\vec{b}}(f_1^k,\dots,f_m^k)(x)\}\|_{l^q}\lesssim\sum_{i=1}^m\mathcal{A}_{m;\mathcal{S},L(\log L)^{\vec{\beta}_i}}\big(\|\{f_1^k\}\|_{l^{q_1}},\dots,\|\{f_m^k\}\|_{l^{q_m}}\big)(x)\\
&&\qquad\qquad+\mathcal{A}_{m;\mathcal{S},\,\vec{b}}\big(\|\{f_1^k\}\|_{l^{q_1}},\dots,\|\{f_m^k\}\|_{l^{q_m}}\big)(x),\nonumber\end{eqnarray}
with $\vec{\beta}_1=(\frac{1}{s_1},\,0,\,\dots,\,0),$ $\dots,\,\vec{\beta}_m=(0,\,\dots,\,0,\,\frac{1}{s_m}).$
As in the proof of Lemma \ref{l3.2}, Theorem \ref{t1.1}, Theorem \ref{t1.2} follows from Theorem \ref{t2.1} and Theorem \ref{t2.2}. We omit the details for brevity.

\section{Proof of Theorem \ref{t1.3}}
For $\beta_1,\,\dots,\,\beta_m\in [0,\,\infty)$, let $\mathcal{M}_{L(\log L)^{\vec{\beta}}}$ be the maximal operator defined by
$$\mathcal{M}_{L(\log L)^{\vec{\beta}}}(f_1,\,\dots,\,f_m)(x)=\sup_{Q\ni x}\prod_{j=1}^m\|f_j\|_{L(\log L)^{\beta_j},\,Q}.$$
For the case of $\vec{\beta}=(0,\,\dots,\,0)$, we denote $\mathcal{M}_{L(\log L)^{\vec{\beta}}}$ by $\mathcal{M}$.
As in \cite{ler} and \cite{perez2}, we can prove that
\begin{lemma}\label{l4.1}
Let $\beta_1,\,\dots,\,\beta_m\in [0,\,\infty)$, $|\beta|=\beta_1+\dots+\beta_m$ and $\vec{w}=(w_1,\,\dots,\, w_m)\in A_{1,\dots,\,1}(\mathbb{R}^{mn})$. Then
for each $\lambda>0$,
\begin{eqnarray*}&&\nu_{\vec{w}}\big(\{x\in\mathbb{R}^n:\, \mathcal{M}_{L(\log L)^{\vec{\beta}}}(f_1,\,\dots,\,f_m)(x)>\lambda\}\big)
\\
&&\quad\lesssim \prod_{j=1}^m\Big(\int_{\mathbb{R}^n}\frac{|f_j(x)|}{\lambda^{\frac{1}{m}}}\log ^{|\beta|}\Big(1+\frac{|f_j(x)|}{\lambda^{\frac{1}{m}}}\Big)w_j(x){\rm d}x
\Big)^{\frac{1}{m}}.
\end{eqnarray*}
\end{lemma}
The following conclusion was established by Lerner et al. in \cite{ler3}.

\begin{lemma}\label{l4.2}
Let $\beta\in [0,\,\infty)$ and $\mathcal{S}$ be a sparse family of cubes. Then for each fixed $\lambda>0$,
$$|\{x\in\mathbb{R}^n:\,\mathcal{A}_{\mathcal{S},\,L(\log L)^{\beta}}f(x)>\lambda\}|\lesssim \int_{\mathbb{R}^n}
\frac{|f(x)|}{\lambda}\log^{\beta}\Big(1+\frac{|f(x)|}{\lambda}\Big){\rm d}x,$$
and for $b\in {\rm BMO}(\mathbb{R}^n)$,
$$|\{x\in\mathbb{R}^n:\,\mathcal{A}_{\mathcal{S},\,b}f(x)>\lambda\}|\lesssim \lambda^{-1}\int_{\mathbb{R}^n}
|f(x)|{\rm d}x.$$
\end{lemma}
\begin{lemma}\label{l4.3}Let $\varrho\in [0,\,\infty)$ and $\delta\in (0,\,1)$, $T$ be a sublinear operator  which satisfies the weak type estimate that
$$|\{x\in\mathbb{R}^n:\, |Tf(x)|>\lambda\}|\lesssim \int_{\mathbb{R}^n}\frac{|f(x)|}{\lambda}\log ^{\varrho}
\Big(1+\frac{|f(x)|}{\lambda}\Big){\rm d}x.$$
Then  for any cube $I$ and appropriate function $f$ with ${\rm supp}\,f \subset I$,
\begin{eqnarray}\Big(\frac{1}{|I|}\int_I|Tf(x)|^{\delta}{\rm d}x\Big)^{\frac{1}{\delta}}\lesssim \|f\|_{L(\log L)^{\varrho},\,I}.\end{eqnarray}
\end{lemma}
For the proof of Lemma \ref{l4.3}, see \cite{guoen}.
\begin{lemma}\label{l4.4}
Let $m\geq 2$ be an integer, $\mathscr{D}$ be a dyadic grid and $\mathcal{S}\subset \mathscr{D}$ be a finite sparse family. Then for
each fixed $I\in\mathscr{D}$ and  $\delta\in (0,\,\frac{1}{m})$.
\begin{eqnarray}&&\inf_{c\in\mathbb{C}}\Big(\frac{1}{|I|}\int_{I}|\mathcal{A}_{m;\,\mathcal{S},\,L(\log L)^{\vec{\beta}}}(f_1,\dots,f_m)(x)-c|^{\delta}{\rm d}x\Big)^{\frac{1}{\delta}}\lesssim_{\delta}
\prod_{j=1}^m\|f_j\|_{L(\log L)^{\beta_j},\,I};\end{eqnarray}
and for $b_i\in Osc_{{\rm exp}L^{s_i}}(\mathbb{R}^n)$ ($s_i\in [1,\,\infty)$, $i=1,\,\dots,\,m$), $\gamma\in (\delta,\,\frac{1}{m})$,
\begin{eqnarray}&&\inf_{c\in\mathbb{C}}\Big(\frac{1}{|I|}\int_{I}|\mathcal{A}_{m;\,\mathcal{S},\,\vec{b}}(f_1,\dots,f_m)(x)-c|^{\delta}{\rm d}x\Big)^{\frac{1}{\delta}}\\
&&\quad\lesssim_{\delta} \inf_{y\in I}M_{\gamma}\big(\mathcal{A}_{m;\,\mathcal{S}}(f_1,\,\dots,\,f_m)\big)(y)+
\prod_{j=1}^m\langle |f_j|\rangle_I.\nonumber\end{eqnarray}
\end{lemma}
\begin{proof} Without loss of generality, we may assume that the functions $f_1,\,\dots,\,f_m$ are nonnegative.
Let $c_0=\sum_{Q\supset I}\prod_{j=1}^m\|f_j\|_{L(\log L)^{\beta_j},\,Q}$. As in \cite{csmp}, it  follows that
\begin{eqnarray*}
&&\int_{I}|\mathcal{A}_{m;\,\mathcal{S},\,L(\log L)^{\vec{\beta}}}(f_1,\dots,f_m)(x)-c_0|^{\delta}{\rm d}x\\
&&\quad\lesssim\int_{I}\Big|\sum_{Q\in\mathcal{S},\,Q\subset I}\|f_j\|_{L(\log L)^{\beta_j},\,Q}\chi_{Q}(x)\Big|^{\delta}{\rm d}x\\
&&\quad\lesssim\int_{I}\Big|\mathcal{A}_{m;\,\mathcal{S},L(\log L)^{\vec{\beta}}}(f_1\chi_I,\dots,\,f_m\chi_{I})(x)\Big|^{\delta}{\rm d}x.
\end{eqnarray*}
On the other hand, by Lemma \ref{l4.2} and Lemma \ref{l4.3}, we know that
\begin{eqnarray*}
\Big(\frac{1}{|I|}\int_{I}\Big|\mathcal{A}_{\mathcal{S},L(\log L)^{\beta_j}}(f_j\chi_I)(x)\Big|^{m\delta}{\rm d}x\Big)^{\frac{1}{m\delta}}\lesssim\|f_j\|_{L(\log L)^{\beta_j},\,I}.
\end{eqnarray*}
This, together with the fact that
$$\mathcal{A}_{m;\,\mathcal{S},\,L(\log L)^{\vec{\beta}}}(f_1,\dots,f_m)(x)\lesssim \prod_{j=1}^m\mathcal{A}_{\mathcal{S},\,L(\log L)^{\beta_j}}f_j(x).
$$
and H\"older's inequality, leads to that
\begin{eqnarray*}
&&\Big(\frac{1}{|I|}\int_{I}\Big|\mathcal{A}_{m;\,\mathcal{S},L(\log L)^{\vec{\beta}}}(f_1\chi_I,\dots,\,f_m\chi_{I})(x)\Big|^{\delta}{\rm d}x\Big)^{\frac{1}{\delta}}\\
&&\quad\lesssim\prod_{j=1}^m\Big(\frac{1}{|I|}\int_{I}\Big|\mathcal{A}_{\mathcal{S},L(\log L)^{\beta_j}}(f_j\chi_I)(x)\Big|^{m\delta}{\rm d}x\Big)^{\frac{1}{m\delta}}\\
&&\quad\lesssim\prod_{j=1}^m\|f_j\|_{L(\log L)^{\beta_j},\,I}.
\end{eqnarray*}

To prove (4.2), we first observe that, for each constant $c\in\mathbb{C}$ and a cube $I\subset \mathscr{D}$,
\begin{eqnarray*}
&&\big|\mathcal{A}_{m;\,\mathcal{S},\,\vec{b}}(f_1,\dots,f_m)(x)-c\big|\\
&&\quad\leq \sum_{i=1}^m|b_i(x)-\langle b_i\rangle_I|\mathcal{A}_{m;\,\mathcal{S}}(f_1,\,\dots,\,f_m)(x)\\
&&\qquad+\Big|
\sum_{Q\in\mathcal{S}}\Big(\sum_{i=1}^m|\langle b_i\rangle_I-\langle b_i\rangle_Q|\Big)\prod_{j=1}^m\langle f_j\rangle_Q\chi_{Q}(x)-c\Big|.
\end{eqnarray*}
Therefore, Let $c_1=\sum_{Q\in\mathcal{S},\,Q\supset I}\Big(\sum_{i=1}^m|\langle b_i\rangle_I-\langle b_i\rangle_Q|\Big)\prod_{j=1}^m\langle f_j\rangle_Q$, we thus have  that
\begin{eqnarray*}&&\inf_{c\in\mathbb{C}}\Big(\frac{1}{|I|}\int_{I}|\mathcal{A}_{m;\,\mathcal{S},\,\vec{b}}(f_1,\dots,f_m)(x)-c|^{\delta}{\rm d}x\Big)^{\frac{1}{\delta}}\\
&&\quad\lesssim\Big(\frac{1}{|I|}\int_{I}\big|\sum_{i=1}^m|b_i(x)-\langle b_i\rangle_I\big|\mathcal{A}_{m;\,\mathcal{S}}(f_1,\,\dots,\,f_m)(x)|^{\delta}{\rm d}x\Big)^{\frac{1}{\delta}}\\
&&\qquad+\Big(\frac{1}{|I|}\int_{I}\Big|
\sum_{Q\in\mathcal{S}}\Big(\sum_{i=1}^m|\langle b_i\rangle_I-\langle b_i\rangle_Q|\Big)\prod_{j=1}^m\langle f_j\rangle_Q\chi_{Q}(x)-c_1\Big|^{\delta}{\rm d}x\Big)^{\frac{1}{\delta}}
\end{eqnarray*}
Let $\gamma\in (\delta,\,\frac{1}{m})$. It follows from H\"older's inequality that
\begin{eqnarray*}
&&\Big(\frac{1}{|I|}\int_{I}\Big|\sum_{i=1}^m\big|b_i(x)-\langle b_i\rangle_I\big|\mathcal{A}_{m;\,\mathcal{S}}(f_1,\,\dots,\,f_m)(x)\Big|^{\delta}{\rm d}x\Big)^{\frac{1}{\delta}}\\
&&\quad\lesssim\Big(\frac{1}{|I|}\int_{I}|\mathcal{A}_{m;\,\mathcal{S}}(f_1,\,\dots,\,f_m)(x)|^{\gamma}{\rm d}x\Big)^{\frac{1}{\gamma}}\\
&&\quad\lesssim \inf_{y\in I}M_{\gamma}\big(\mathcal{A}_{m;\,\mathcal{S}}(f_1,\,\dots,\,f_m)\big)(y).
\end{eqnarray*}
On the other hand, we deduce from H\"older's inequality, Lemma \ref{l4.2} and Lemma \ref{l4.3},  that
\begin{eqnarray*}&&\Big(\frac{1}{|I|}\int_{I}\Big|
\sum_{Q\in\mathcal{S},\,Q\subset I}\Big(\sum_{i=1}^m| \langle b_i\rangle_I-\langle b_i\rangle_Q|\Big)\prod_{j=1}^m\langle f_j\chi_{I}\rangle_Q\chi_{Q}(x)\Big|^{\delta}{\rm d}x\Big)^{\frac{1}{\delta}}\\
&&\quad\lesssim \Big(\frac{1}{|I|}\int_{I}\big(\mathcal{A}_{m;\,\mathcal{S},\,\vec{b}}(f_1\chi_{I},\dots,\,f_m\chi_I)(x)\big)^{\delta}{\rm d}x\Big)^{\frac{1}{\delta}}\\
&&\qquad+\Big(\frac{1}{|I|}\int_{I}\Big(\sum_{i=1}^m| b_i(x)-\langle b_i\rangle_I|\Big)\big (\mathcal{A}_{m;\,\mathcal{S}}(f_1\chi_I,\dots,\,f_m\chi_I)(x)\big)^{\delta}{\rm d}x\Big)^{\frac{1}{\delta}}\\
&&\quad\lesssim \prod_{j=1}^m\langle |f_j|\rangle_I.
\end{eqnarray*}
Combining the estimates above leads to (4.2).
\end{proof}

Let $\mathscr{D}$ be a dyadic grid. Associated with $\mathscr{D}$, define the sharp maximal function $M^{\sharp}_{\mathscr{D}}$ as
$$M^{\sharp}_{\mathscr{D}}f(x)=\sup_{Q\ni x\atop{Q\in\mathscr{D}}}\inf_{c\in\mathbb{C}}\frac{1}{|Q|}\int_{Q}|f(y)-c|{\rm d}y.$$
For $\delta\in (0,\,1)$, let $ M_{\mathscr{D},\,\delta}^{\sharp}f(x)=\big[M^{\sharp}_{\mathscr{D}}(|f|^{\delta})(x)\big]^{1/\delta}.$
Repeating the argument in \cite[p. 153]{ste2}, we can verify that if $u\in A_{\infty}(\mathbb{R}^n)$ and $\Phi$ is a increasing function on $[0,\,\infty)$ which satisfies that
$$\Phi(2t)\leq C\Phi(t),\,t\in [0,\,\infty),$$ then
\begin{eqnarray}&&\sup_{\lambda>0}\Phi(\lambda)u(\{x\in\mathbb{R}^n:|h(x)|>\lambda\})\lesssim
\sup_{\lambda>0}\Phi(\lambda)u(\{x\in\mathbb{R}^n:M_{\mathscr{D},\delta}^{\sharp}h(x)>\lambda\}),\end{eqnarray}
provided that $\sup_{\lambda>0}\Phi(\lambda)u(\{x\in\mathbb{R}^n:\,M_{\mathscr{D},\,\delta}h(x)>\lambda\})<\infty$.

\medskip

{\it Proof of Theorem \ref{t1.3}}. Let $\vec{\beta}_1=(\frac{1}{s_1},\,0,\dots,\,0)$, $\dots$, $\vec{\beta}_m=(0,\dots,0,\,\frac{1}{s_m})$. By the inequality (3.7) and the one-third trick, it suffices to prove that
for $\vec{w}=(w_1,\,\dots,\,w_m)\in A_{1,\,\dots,1}(\mathbb{R}^{mn})$,  $i=1,\,\dots,\,m$, dyadic grid $\mathscr{D}$ and sparse  family $\mathcal{S}\subset\mathscr{D}$,
\begin{eqnarray}&&\nu_{\vec{w}}(\{x\in\mathbb{R}^n:\,\mathcal{A}_{m;\,\mathcal{S},L(\log L)^{\vec{\beta}_i}}(f_1,\,\dots,\,f_m)(x)>1\})\\
&&\quad\lesssim\prod_{j=1}^m\Big(\int_{\mathbb{R}^n}
|f_j(y_j)|\log^{\frac{1}{s_i}}\big(1+|f_j(y_j)|\big)w_j(y_j){\rm d}y_j\Big)^{\frac{1}{m}},\nonumber
\end{eqnarray}
and
\begin{eqnarray}&&\nu_{\vec{w}}(\{x\in\mathbb{R}^n:\,\mathcal{A}_{m;\,\mathcal{S},\,\vec{b}}(f_1,\,\dots,\,f_m)(x)>1\})\lesssim\prod_{j=1}^m\|f_j\|_{L^{p_j}(\mathbb{R}^n,\,w_j)}^{\frac{1}{m}}.
\end{eqnarray}

We first prove (4.5). By a standard limit argument, it suffices to consider the case that the sparse family $\mathcal{S}$
is finite. Let $\delta\in (0,\,\frac{1}{m})$. The estimate (4.2) in Lemma \ref{l4.4} tells us that
$$M^{\sharp}_{\mathscr{D},\,\delta}\big(\mathcal{A}_{m;\,\mathcal{S},L(\log L)^{\vec{\beta}_i}}(f_1,\,\dots,\,f_m)\big)(x)
\lesssim \mathcal{M}_{L(\log L)^{\vec{\beta}_i}}(f_1,\,.\,\dots,\,f_m)(x).$$
Let $\psi_i(t)=t^{\frac{1}{m}}\log^{-\frac{1}{s_i}}(1+t^{-\frac{1}{m}})$. Lemma \ref{l4.1} now tells us that
\begin{eqnarray*}&&\sup_{t>0}\psi_i(t)\nu_{\vec{w}}\big(\{x\in\mathbb{R}^n:\,\mathcal{M}_{L(\log L)^{\vec{\beta}_i}}(f_1,\,\dots,\,f_m)\big)(x)>t\}\big)\\
&&\quad\lesssim
\prod_{j=1}^m\Big(\int_{\mathbb{R}^n}
|f_j(y_j)|\log^{\frac{1}{s_i}}\big(1+|f_j(y_j)|\big)w_j(y_j){\rm d}y_j\Big)^{\frac{1}{m}}.\end{eqnarray*}
This, via (4.4) and Lemma 4.1, implies that
\begin{eqnarray*}&&\nu_{\vec{w}}(\{x\in\mathbb{R}^n:\,\mathcal{A}_{m;\,\mathcal{S},\,L(\log L)^{\vec{\beta}_i}}(f_1,\dots,f_m)(x)>1\})\\
&&\quad\lesssim \sup_{t>0}\psi_i(t)\nu_{\vec{w}}\big(\{x\in\mathbb{R}^n:\,M^{\sharp}_{\mathscr{D},\,\delta}\big(\mathcal{A}_{m;\,\mathcal{S},\,L(\log L)^{\vec{\beta}_i}}(f_1,\dots,f_m)\big)(x)>t\}\big)\\
&&\quad\lesssim\prod_{j=1}^m\Big(\int_{\mathbb{R}^n}
|f_j(y_j)|\log^{\frac{1}{s_i}}\big(1+|f_j(y_j)|\big)w_j(y_j){\rm d}y_j\Big)^{\frac{1}{m}}.\nonumber
\end{eqnarray*}

We turn our attention to (4.6). Again we assume that the the sparse family $\mathcal{S}$ is finite.
Applying Lemma \ref{l4.4}, we see that for $\delta,\,\gamma$ with $0<\delta<\gamma<\frac{1}{m}$,
$$M^{\sharp}_{\mathscr{D},\,\delta}\big(\mathcal{A}_{m;\,\mathcal{S},\,\vec{b}}(f_1,\,\dots,\,f_m)\big)(x)\lesssim M_{\gamma}\big(\mathcal{A}_{m;\,\mathcal{S}}(f_1,\,\dots,\,f_m)\big)(x)+\mathcal{M}(f_1,\dots,f_m)(x).
$$
Recalling that $\nu_{\vec{w}}\in A_{\infty}(\mathbb{R}^{mn})$, we can choose $\delta$ and $\gamma$ in (4.3) small enough such that $\nu_{\vec{w}}\in A_{\frac{1}{m\gamma}}(\mathbb{R}^{mn})$.
It then follows from Lemma \ref{l3.0}, the inequality (4.2) and Lemma \ref{l4.1} that
\begin{eqnarray*}
&&\lambda^{\frac{1}{m}}\nu_{\vec{w}}(\{x\in\mathbb{R}^n:\,M_{\gamma}\big(\mathcal{A}_{m;\,\mathcal{S}}(f_1,\,\dots,\,f_m)\big)(x)>\lambda\})\\
&&\quad\lesssim \sup_{t>0}t^{\frac{1}{m}}\nu_{\vec{w}}(\{x\in\mathbb{R}^n:\,\mathcal{A}_{m;\,\mathcal{S}}(f_1,\,\dots,\,f_m)(x)>t\})\\
&&\quad\lesssim \sup_{t>0}t^{\frac{1}{m}}\nu_{\vec{w}}(\{x\in\mathbb{R}^n:\,M^{\sharp}_{\mathscr{D},\,\delta}\big(\mathcal{A}_{m;\,\mathcal{S}}(f_1,\,\dots,\,f_m)\big)(x)>t\})\\
&&\quad\lesssim \sup_{t>0}t^{\frac{1}{m}}\nu_{\vec{w}}(\{x\in\mathbb{R}^n:\,\mathcal{M}(f_1,\,\dots,\,f_m)(x)>t\})\\
&&\quad\lesssim\prod_{j=1}^m\|f_j\|_{L^{p_j}(\mathbb{R}^n,\,w_j)}^{\frac{1}{m}},
\end{eqnarray*}
This, toether with (4.4),leads to  that
\begin{eqnarray*}&&\nu_{\vec{w}}(\{x\in\mathbb{R}^n:\,\mathcal{A}_{m;\,\mathcal{S},\,\vec{b}}(f_1,\,\dots,\,f_m)(x)>1\})\\
&&\quad\lesssim \sup_{t>0}t^{\frac{1}{m}}\nu_{\vec{w}}\big(\{x\in\mathbb{R}^n:\,M^{\sharp}_{\mathscr{D},\,\delta}\big(\mathcal{A}_{m;\,\mathcal{S},\,\vec{b}}(f_1,\,\dots,\,f_m)\big)(x)>t\}\big)\\
&&\quad\lesssim\prod_{j=1}^m\|f_j\|_{L^{p_j}(\mathbb{R}^n,\,w_j)}^{\frac{1}{m}},\nonumber
\end{eqnarray*}
and then  completes the proof of Theorem \ref{t1.3}.\qed

\section{Applications to the Commutators of Calder\'on}

Let us consider the $m$-th commutator of Calder\'on, which is defined by
$$\mathcal{C}_{m+1}(a_1,\,\dots,\,a_m,\,f)(x)={\rm p.\,v.}\,\int_{\mathbb{R}^n}\frac{\prod_{j=1}^m(A_j(x)-A_j(y))}{(x-y)^{m+1}}f(y){\rm d}y,
$$
where $a_j=A_j'$. This operator first appeared in the study
of Cauchy integrals along Lipschitz curves and, in fact, led to the first proof of the $L^2$
boundedness of the latter.

When $m = 1$, it is well known that $\mathcal{C}_2$ is bounded from $L^{p_1}(\mathbb{R})\times L^{p_2}(\mathbb{R})$ to $L^p(\mathbb{R})$ when
$1<p_1,\,p_2\leq \infty$ and $\frac{1}{2}<p\leq\infty$ satisfying $1/p=1/p_1+1/p_2$; and moreover, it is bounded from
$L^{p_1}(R)\times L^{p_2}(\mathbb{R})$ to $L^{p,\,\infty}(\mathbb{R})$ if $\min\{p_1,\,p_2\}=1$ and in particular it is bounded from
$L^1(\mathbb{R})\times L^1(\mathbb{R})$ to $L^{\frac{1}{2}}(\mathbb{R})$; see \cite{ca1,ca2}. The corresponding result that $\mathcal{C}_3$ maps
$L^1(\mathbb{R})\times L^1(\mathbb{R})\times L^1(\mathbb{R})$ to $L^{\frac{1}{3},\,\infty}(\mathbb{R})$ was proved by Coifman and Meyer; see \cite{cm2}, while the analogous
result for  $\mathcal{C}_{m+1}$, $m\geq 3$, was established by Duong, Grafakos, and Yan \cite{dgy}.
As it was proved in \cite{dgy}, $\mathcal{C}_{m+1}$ can be rewritten as the folloing multilinear singular integral operator
\begin{eqnarray}&&\mathcal{C}_{m+1}(a_1,\,\dots,\,a_m,\,f)(x)\\
&&\quad=\int_{\mathbb{R}^{m+1}}K(x; y_1,\dots, y_{m+1})\prod_{j=1}^ma_j(y_j)f(y_{m+1}) dy_1 \dots dy_{m+1};\nonumber
\end{eqnarray}
with
\begin{eqnarray*}
K(x;y_1,\dots, y_{m+1}) =
\frac{(-1)^{me(y_{m+1}-x)}}{(x-y_{m+1})^{m+1}}
\prod_{j=1}^m
\chi_{(\min\{x, y_{m+1}\}, \max\{x, y_{m+1}\})}(y_j),
\end{eqnarray*}
and $e$ is the characteristic function of $[0,\,\infty)$. Using some new maximal operators, Grafakos, Liu and Yang \cite{gly} proved that if $p_1,\,\dots,\,p_{m+1}\in [1,\,\infty)$ and $p\in [\frac{1}{m+1},\,\infty)$ with $1/p=1/p_1+\dots+1/p_{m+1}$,  and $\vec{w}=(w_1,\,\dots,\,w_m, w_{m+1})\in A_{\vec{P}}(\mathbb{R}^{m+1})$, then $\mathcal{C}_{m+1}$ is bounded
from $L^{p_1}(\mathbb{R},\,w_1)\times \dots\times L^{p_{m+1}}(\mathbb{R},\,w_m)$ to $L^{p,\,\infty}(\mathbb{R}^n, \nu_{\vec{w}})$, and when $\min_{1\leq j\le m+1}p_j>1$, $\mathcal{C}_{m+1}$ is bounded from $L^{p_1}(\mathbb{R}, w_1)\times \dots\times L^{p_{m+1}}(\mathbb{R},w_{m+1})$ to $L^{p}(\mathbb{R},\,\nu_{\vec{w}})$. It was pointed out in \cite{huz} that $\mathcal{C}_{m+1}$ satisfies Assumption \ref{a1.1} and (1.11). Thus by Theorems \ref{t1.1}, \ref{t1.2} and \ref{t1.3}, we have the following conclusions.
\begin{corollary}\label{c1.1}
Let $m\geq 1$, $p_1,\dots,p_{m+1},\,q_1,\dots,q_{m+1}\in (1,\,\infty)$, $p,q\in (\frac{1}{m+1},\infty)$ with $1/p=1/p_1+\dots+1/p_{m+1}$, $1/q=1/q_1+\dots+1/q_{m+1}$, $\vec{w}=(w_1,\dots,w_{m+1})\in A_{\vec{P}}(\mathbb{R}^{m+1})$. Then \begin{eqnarray*}&&\|\{\mathcal{C}_{m+1}(a_1^k,\dots,a_m^k,\,f^k)\}\|_{L^p(l^q;\mathbb{R}^n,\nu_{\vec{w}})}\lesssim[\vec w]_{A_{\vec P}}^{\max \{1, \frac{p_1'}{p},\cdots, \frac{p'_{m+1}}{p}\}}\\
&&\qquad\qquad\times\prod_{j=1}^m\|\{a_j^k\}\|_{L^{p_j}
(l^{q_j};\,\mathbb{R},w_j)}\|\{f^k\}\|_{L^{p_{m+1}}(l^{q_{m+1}};\,\mathbb{R},w_{m+1})}.\end{eqnarray*}
\end{corollary}
\begin{corollary}Let $m\geq 1$, $p_1,\dots,p_{m+1},\,q_1,\,\dots,q_{m+1}\in (1,\,\infty)$, $p,\,q\in (\frac{1}{m+1},\infty)$ with $1/p=1/p_1+\dots+1/p_{m+1}$, $1/q=1/q_1+\dots+1/q_{m+1}$, $\vec{w}=(w_1,\,\dots,\,w_{m+1})\in A_{\vec{P}}(\mathbb{R}^{m+1})$.
Let $b_j\in Osc_{{\rm exp}L^{s_j}}(\mathbb{R})$  with $\sum_{j=1}^{m+1}\|b_j\|_{Osc_{{\rm exp}L^{s_j}}(\mathbb{R})}=1$.  Then $\mathcal{C}_{m+1,\,\vec{b}}$, the commutator of $\mathcal{C}_{m+1}$ defined as (1.12), satisfies the weighted estimate  that
\begin{eqnarray*}&&\|\{\mathcal{C}_{m+1,\,\vec{b}}(a_1^k,\dots,a_m^k,\,f^k)\}\|_{L^p(l^q;\mathbb{R}^n,\nu_{\vec{w}})}\lesssim[\vec w]_{A_{\vec P}}^{\max \{1, \frac{p_1'}{p},\cdots, \frac{p'_{m+1}}{p}\}}\\
&&\quad\times\Big([\nu_{\vec{w}}]_{A_{\infty}}^{\frac{1}{s_*}}+\sum_{i=1}^m[\sigma_i]_{A_{\infty}}^{\frac{1}{s_i}}\Big)\prod_{j=1}^m\|\{a_j^k\}\|_{L^{p_j}
(l^{q_j};\,\mathbb{R},w_j)}\|\{f^k\}\|_{L^{p_{m+1}}(l^{q_{m+1}};\,\mathbb{R},w_{m+1})}.\end{eqnarray*}
\end{corollary}
\begin{corollary}Let $m\geq 1$, $q_1,\,\dots,q_{m+1}\in (1,\,\infty)$, $q\in (1/(m+1),\infty)$ with $1/p=1/p_1+\dots+1/p_{m+1}$, $1/q=1/q_1+\dots+1/q_{m+1}$, $\vec{w}=(w_1,\,\dots,\,w_{m+1})\in A_{\vec{P}}(\mathbb{R}^{m+1})$.
Let $b_j\in Osc_{{\rm exp}L^{s_j}}(\mathbb{R})$ $(1\leq j\leq m+1)$ with $\sum_{j=1}^{m+1}\|b_j\|_{Osc_{{\rm exp}L^{s_j}}(\mathbb{R})}=1$. Then for each $\lambda>0$,
\begin{eqnarray*}&&\nu_{\vec{w}}(\{x\in\mathbb{R}^n:\,\|\{\mathcal{C}_{m+1,\,\vec{b}}(a_1^k,\dots,a_m^k,\,f^k)(x)\}\|_{l^q}>\lambda\})\\
&&\quad\lesssim\prod_{j=1}^m\Big(\int_{\mathbb{R}^{n}}
\frac{\|\{a_j^k(y_j)\}\|_{l^{q_j}}}{\lambda^{\frac{1}{m+1}}}\log^{\frac{1}{s_*}}\Big(1+\frac{\|\{a_j^k(y_j)\}\|_{l^{q_j}}}
{\lambda^{\frac{1}{m+1}}}\Big)w_j(y_j){\rm d}y_j\Big)^{\frac{1}{m+1}}\\
&&\qquad\times\Big(\int_{\mathbb{R}^{n}}
\frac{\|\{f^k(y)\}\|_{l^{q_j}}}{\lambda^{\frac{1}{m+1}}}\log^{\frac{1}{s_*}}\Big(1+\frac{\|\{f^k(y)\}\|_{l^{q_j}}}
{\lambda^{\frac{1}{m+1}}}\Big)w_{m+1}(y){\rm d}y\Big)^{\frac{1}{m}}.\end{eqnarray*}
\end{corollary}

\medskip

{\bf Added in Proof}. After this paper was prepared, we learned that Dr. Kangwei Li \cite{li} also observed that, Lerner's idea in \cite{ler2} applies to the multilinear singular integral operators. We remark that our argument in the proof of Theorem \ref{t3.1} also based on this observation. Li \cite{li} proved that the multilinear singular integral operators whose kernels satisfy $L^r$- H\"ormander condition can be dominated by sparse operators. The main results in \cite{li} are different from the results in this paper and are of independent interest.


\begin{thebibliography}{99}
\bibitem{bcd}T. A. Bui, J. Conde-Alonso, X. T. Duong and M. Hormozi, Weighted bounds for multilinear operators with non-smooth kernels,
Arxiv:1506.07752.
\bibitem{buiduong}T. A. Bui and X. T. Doung, On commutators of vector BMO functions and multilinear singular integrals with non-smooth
kernels, J. Math. Anal. Appl. \textbf{371} (2010), 80-94.
\bibitem{bu}S. M. Buckley, Estimates for operator norms on weighted spaces and reverse Jensen inequalities, Trans. Amer. Math. Soc. \textbf{340} (1993), 253-272.
\bibitem{ca1} A. P. Calder\'on, Commutators of singular integral operators, Proc. Nat. Acad. Sci.
U.\,S.\,A. \textbf{53} (1965), 1092-1099.
\bibitem{ca2} C. P. Calder\'on, On commutators of singular integrals, Studia Math. \textbf{53} (1975),
139-174.
\bibitem{chpp}  D. Chung, M. Pereyra, and C. P\'erez, Sharp bounds for general commutators
on weighted Lebesgue spaces, Trans. Amer. Math. Soc.  364 (2012),
1163-1177.
\bibitem{car}J. M. Conde-Alonso and G. Rey, A pointwise estimate for positive dyadic shifts and some applications,
Math. Ann. 2015, DOI 10.1007/s00208-015-1320-y.
\bibitem{cm1} R. R. Coifman and Y.  Meyer,  On commutators of
singular integrals and bilinear singular integrals,  {Trans. Amer.
Math. Soc.}, \textbf{212} (1975), 315-331.
\bibitem{cm2}R. R. Coifman and Y.  Meyer,   Au del$\grave{\rm a}$ des op\'erateurs
pseudo-diff\'erentiels, {Ast\'eriaque} \textbf{57}, 1978.
1-185.
\bibitem{csmp} D. Cruz-Uribe, SFO, J. Martell and C. P\'erez, Sharp weighted estimates for classical operators, Adv. Math.
 \textbf{229} (2012), 408--441.
\bibitem{dhli} W. Dami\'an, M. Hormozi and and K. Li, New bounds for bilinear Calder\'on-Zygmund operators and applications, arxiv:1512.02400.
\bibitem{dra}O. Dragicevi\'c, L. Grafakos, M. Pereyra and S. Petermichl, Extrapolation and sharp norm estimates for classical operators on weighted Lebesgue spaces, Publ. Mat. \textbf{49} (2005), 73-91.

\bibitem{dggy} X. Duong, R. Gong, L. Grafakos, J. Li and L. Yan, Maximal operator for multilinear singular integrals with non-smooth kernels.
Indiana Univ. Math. J. \textbf{58} (2009),  2517-2542.
\bibitem{dgy}X. Duong, L. Grafakos, L. Yan, Multilinear operators with non-smooth kernels and commutators of singular integrals.
Trans. Amer. Math. Soc. \textbf{362} (2010),  2089-2113.
\bibitem{fes} C. Fefferman  and E. M. Stein,  Some maximal operators, Amer. J. Math. \textbf{93} (1971), 107-115.
\bibitem{gk}  L. Grafakos and N. Kalton, Multilinear Calder\'on-Zygmund operators on Hardy spaces, Collect. Math.  \textbf{52} (2001)
169-179.
\bibitem{gly} L. Grafakos, L. Liu and D. Yang,  Multilple-weighted
norm  inequalities for maximal singular integrals with non-smooth
kernels, Proc. Royal Soc. Edinb. \textbf{141A} (2011), 755-775.
\bibitem{gra}  L. Grafakos,  Modern Fourier analysis, GTM250, 2nd
Edition, Springer, New York, 2008.
\bibitem{gt2} L. Grafakos and R. Torres, Multilinear Calder\'on-Zygmund theory,
{ Adv. Math.} \textbf{165} (2002), 124-164.

\bibitem{guoen} G. Hu and D. Li, A Cotlar type inequality for the multilinear singular
integral operators and its applications, J. Math. Anal. Appl. \textbf{290} (2004),  639-653.
\bibitem{huli}G. Hu and K. Li, Weighted vector-valued inequalities for a class of multilinear singular integral operators, Proc. Edinb. Math. Soc., to appear.
\bibitem{huz} G. Hu and Y. Zhu, Weighted norm inequalities with general weights for the commutator of Calder\'on, Acta Math. Sinica, English Ser.  \textbf{29} (2013), 505-514.
\bibitem{hyt} T. Hyt\"onen, The sharp weighted bound for general Calder\'on-Zygmund operators,
Ann. of Math. \textbf{175} (2012), 1473-1506.
\bibitem{hlp}T.  Hyt\"onen, M.  Lacey and C. P\'erez, Sharp weighted bounds for the q-
variation of singular integrals, Bull. Lond. Math. Soc. \textbf{45} (2013), 529-
540.
\bibitem{hp} T. Hyt\"onen and C. P\'erez, Sharp weighted bounds  involving $A_{\infty}$, Anal. PDE. \textbf{6} (2013), 777-818.
\bibitem{ler2} A. K. Lerner, On pointwise estimate involving sparse operator, New York J. Math. 22(2016), 341-349.
\bibitem{ler}A. Lerner,   S. Ombrossi,  C. P\'erez,  R. H. Torres and R. Trojillo-Gonzalez, New maximal functions and multiple
weights for the multilinear Calder\'on-Zygmund theorey,  {Adv. Math.}
\textbf{220} (2009), 1222-1264.
\bibitem{ler3} A. K. Lerner, S. Obmrosi and I. Rivera-Rios, On pointwise and weighted estimates for commutators of Calder\'on-Zygmund operators, arxiv:1604.01334.
\bibitem{li} K. Li, Sparse domination theorem for mltilinear singular integral operators with $L^r$-H\"ormander condition, arxiv:1606.03952.
\bibitem{lms} K. Li, K. Moen and W. Sun, The sharp weighted bound for multilinear maximal
functions and Calder\'on-Zygmund operators, J. Four. Anal. Appl. \textbf{20} (2014), 751-765.
\bibitem{ls} K. Li and W. Sun, Weak and strong type weighted estimates for multilinear
Calder\'on-Zygmund operators, Adv. Math. \textbf{254} (2014), 736-771.
\bibitem{mu} B. Muckenhoupt, Weighted norm inequalities for the Hardy maximal function, Trans.
Amer. Math. Soc. \textbf{165} (1972), 207-226.
\bibitem{perez2} C. P\'erez, G. Pradolini, R. H. Torres and  R. Trujillo-Gonz\'alez, End-point estimates for iterated commutators of multilinear singular integrals, Bull. London Math. Soc. \textbf{46} (2014), 26-42.
\bibitem{pt} C. P\'erez and R. H. Torres, Sharp maximal function estimates for multilinear singular integrals, Contemp. Math.
\textbf{320} (2003), 323-331.
\bibitem{pet1}S. Petermichl, The sharp bound for the Hilbert transform on weighted Lebesgue spaces in terms of
the classical $A_p$ characteristic, Amer. J. Math. \textbf{129} (2007), 1355-1375.
\bibitem{pet2}S. Petermichl, The sharp weighted bound for the Riesz transforms, Proc. Amer. Math. Soc. \textbf{136}
(2008), 1237-1249.
\bibitem{rr}M.  Rao and Z. Ren, Theory of Orlicz spaces, Monographs and Textbooks in Pure
and Applied Mathematics, 146, Marcel Dekker Inc., New York, 1991.

\bibitem{ste2} E. M. Stein, Harmonic Analysis, Real Variable Methods, Orthogonality, and Oscillatory
Integrals, Princeton Univ. Press, Princeton, NJ. 1993.

\end{thebibliography}
\end{document}